\documentclass{amsart}
\usepackage[utf8]{inputenc}
\usepackage{amssymb}
\usepackage{amsfonts}
\usepackage{amssymb}
\usepackage{amsmath}
\usepackage{amsthm}
\usepackage{enumerate}
\usepackage{tabularx}
\usepackage{centernot}
\usepackage{mathtools}
\usepackage{stmaryrd}
\usepackage{amsthm,amssymb}
\usepackage{etoolbox,color}
\usepackage{tikz}
\usepackage{amssymb}
\usetikzlibrary{matrix}
\usepackage{tikz-cd}
\usepackage{tikz}
\usepackage{marginnote}
\definecolor{mygray}{gray}{0.85}
\usepackage[linecolor=black, backgroundcolor=mygray,colorinlistoftodos,prependcaption,textsize=small]{todonotes}
\usepackage[colorlinks,citecolor=blue,urlcolor=blue,linkcolor=blue,hypertexnames=false]{hyperref}
\usepackage{IEEEtrantools}

\renewcommand{\leq}{\leqslant}
\renewcommand{\geq}{\geqslant}

\renewcommand{\trianglelefteq}{\trianglelefteqslant}
\newcommand{\mrm}[1]{\mathrm{#1}}

\usepackage[backgroundcolor=mygray,colorinlistoftodos,prependcaption,textsize=small]{todonotes}
\usepackage{xcolor}
\usepackage{booktabs}

\makeatletter
\def\subsection{\@startsection{subsection}{3}%
  \z@{.5\linespacing\@plus.7\linespacing}{.3\linespacing}%
  {\bfseries\centering}}
\makeatother

\makeatletter
\def\subsubsection{\@startsection{subsubsection}{3}%
  \z@{.5\linespacing\@plus.7\linespacing}{.3\linespacing}%
  {\centering}}
\makeatother

\makeatletter
\def\myfnt{\ifx\protect\@typeset@protect\expandafter\footnote\else\expandafter\@gobble\fi}
\makeatother

\newtheorem{theorem}{Theorem}[section]

\newtheorem{corollary}[theorem]{Corollary}
\newtheorem{lemma}[theorem]{Lemma}

\newtheorem{proposition}[theorem]{Proposition}

\newtheorem{question}[theorem]{Question}

\newtheorem{conjecture}[theorem]{Conjecture}
\newtheorem*{theorem1.3}{Theorem~1.3}
\newtheorem*{corollary1.5}{Corollary~1.5}

\theoremstyle{plain}

\theoremstyle{definition}

\newtheorem{fact}[theorem]{Fact}
\newtheorem{definition}[theorem]{Definition}
\newtheorem{remark}[theorem]{Remark}

\newtheorem{notation}[theorem]{Notation}

\newcounter{claimcounter}

\usepackage{xcolor}
\definecolor{mygray}{gray}{0.85}
\usepackage[backgroundcolor=mygray,colorinlistoftodos,prependcaption,textsize=small]{todonotes}

\presetkeys{todonotes}{}{nolist}
\begin{document}


\begin{abstract} 
We prove several results on the model theory of Artin groups, focusing on Artin groups which are ``far from right-angled Artin groups''. The first result is that if $\mathcal{C}$ is a class of Artin groups whose irreducible components are acylindrically hyperbolic and torsion-free, then the model theory of Artin groups of type $\mathcal{C}$ reduces to the model theory of their irreducible components. The second result is that the problem of superstability of a given non-abelian Artin group $A$ reduces to certain dihedral parabolic subgroups of $A$ being $n$-pure in $A$, for certain large enough primes  $n \in \mathbb{N}$. The third result is that two spherical Artin groups are elementarily equivalent if and only if they are isomorphic. Finally, we prove that the affine Artin groups of type $\tilde{A}_n$, for $n \geq 4$, can be distinguished from the other simply laced affine Artin groups using existential sentences; this uses known computations of second integral homology together with a direct calculation for simply laced Artin groups under the $K(\pi,1)$ assumption. These computations provide obstructions to the retractions forced by existential equivalence.
%
%
\end{abstract}

\title{First-order aspects of Artin groups}


\thanks{Research of the first and second named authors was  supported by project PRIN 2022 ``Models, sets and classifications", prot.~2022TECZJA, and of the second named author by INdAM Project 2024 (Consolidator grant) ``Groups, Crystals and Classifications''. The third named author was supported by project PRIN 2022 ``Algebraic and topological combinatorics'', prot.~2022A7L229. The first and third named authors are affiliated with the GNSAGA research groups of INdAM. The authors are deeply grateful to Simon Andr{\' e}  for his essential contributions to the proofs of the following three claims: \ref{main_th1}, \ref{simon_prop} and \ref{affArt_retr}}


\author{Alberto Cassella}
\address{Department of Mathematics ``Giuseppe Peano'', University of Torino, Via Carlo Alberto 10, 10123, Italy.}
\email{alberto.cassella@unito.it}

\author{Gianluca Paolini}
\address{Department of Mathematics ``Giuseppe Peano'', University of Torino, Via Carlo Alberto 10, 10123, Italy.}
\email{gianluca.paolini@unito.it}

\author{Giovanni Paolini}
\address{Department of Mathematics, University of Bologna, Piazza di Porta S.~Donato 5, 40126, Italy}
\email{g.paolini@unibo.it}

\date{\today}
\maketitle




\section{Introduction}

	In recent years, the model theory of finitely generated groups has seen crucial advancements. A prominent role in these advancements has been played by the employment of methods from hyperbolic geometry and, more generally, geometric group theory. Two classes of finitely generated groups central to geometric group theory are certainly Coxeter groups and their cousins, Artin groups. In \cite{MPS22}, the second author et al. started investigating the (first-order) model theory of (finitely generated) Coxeter groups;	this was later continued in \cite{PaAnd, PaAnd2, PaSk}, with special focus on affine and hyperbolic Coxeter groups.
	In another direction, Casals-Ruiz, Kazachkov, et al., proved many beautiful results on the model theory of right-angled Artin groups (RAAGs); see e.g.\ \cite{CRKR10, CRK11, CRKNG}. The aim of this paper is to advance the model theory of Artin groups, focusing on results that involve other classes of Artin groups, with a particular focus on groups which are ``far \mbox{from RAAGs''.} 
	
\medskip
	
	We structure our paper around four fundamental questions: 
    \begin{enumerate}[(A)]
        \item domains and irreducible components;
        \item superstability;
        \item distinguishing spherical Artin groups via first-order sentences;
        \item distinguishing affine Artin groups via first-order sentences.
    \end{enumerate}

\subsection{Reduction to irreducible components}\label{irr_component_intro}

	Artin groups, as well as Coxeter groups, are defined using certain labeled graphs. More precisely, given a finite complete graph $\Gamma$ on a set $S$ with labels $m(s, t)$ from $\{n \in \mathbb{N} : n \geq 2\} \cup \{\infty\}$, one defines the Artin group $A(\Gamma)$ through the following presentation:
	$$\langle S \mid (s, t)_{m(s, t)} = (t, s)_{m(s, t)}, \, s, t \in S \rangle,$$
where $(s, t)_{m(s, t)}$ denotes the alternating word on alphabet $\{s, t\}$ which starts with $s$ and has length ${m(s, t)}$; by convention, if ${m(s, t)} = \infty$, then we let $(s, t)_{m(s, t)} = e$. The Coxeter group $W(\Gamma)$ associated to $\Gamma$ is defined by the same presentation as $A(\Gamma)$, with the extra relations that make every generator an involution.

	One often drops the edges labeled by $2$ from the graph $\Gamma$, and in this way the connected components of $\Gamma$ correspond to direct summands of the group $A$, and so, given an Artin system $(A, S)$, there exists a unique decomposition of $A$ into irreducible components $A_1 \times \cdots \times A_n$. Anybody familiar with Artin groups (resp. Coxeter groups) knows that in most cases the study of these groups reduces to the study of their irreducible components. In this spirit, we might wonder: does the model theory of an Artin group reduce to the model theory of its irreducible components?
    Thankfully, there is a well-developed general theory which can be employed to answer this question, which goes under the name of {\em algebraic geometry over groups} (cf.\ \cite{alg_geom_over_groups, alg_geom_over_groups2, alg_geom_over_groups3}). In particular, in \cite{alg_geom_over_groups}, the following notion  was isolated: we say that a group $A$ is a domain if for every $x, y \in A$, whenever $x, y \neq e$ there is $g \in A$ such that $[x, y^g] \neq e$. In \cite{alg_geom_over_groups3}, it was then shown that if a group $A$ decomposes as a direct product $A_1 \times \cdots \times A_n$ of domains, then, in many respects, the model theory of $A$ reduces to the model theory of the components $A_i$. More precisely we have:
	
	\begin{fact}[{\cite{alg_geom_over_groups3}}]\label{reduction_to_components_fact} Let $A = A_1 \times \cdots \times A_n$ with each $A_i$ a nontrivial domain. Then:
	\begin{enumerate}[(i)]
	\item if $A \equiv B$, then $B$ is also a finite direct product of nontrivial domains $B = B_1 \times \cdots \times B_k$, with $k = n$ and, after a suitable reordering of the factors, $A_i \equiv B_i$ for all $i \in [1, n]$;

	\item for every $i \in [1, n]$, $\mrm{Th}(A_i)$ is interpretable in $A$;
	\item $A$ is $\lambda$-stable if and only if $A_i$ is $\lambda$-stable for every $i \in [1, n]$;
	\item $\mrm{Th}(A)$ is decidable if and only if $\mrm{Th}(A_i)$ is decidable for every $i \in [1, n]$.
\end{enumerate}
\end{fact}
	
	Thus, the question of reduction of the model theory of a given Artin group to its irreducible components reduces to understanding which Artin groups are domains, in the context of algebraic geometry over groups. Clearly, in order to be a domain, a group has to be indecomposable and have a trivial center. We conjecture that for Artin groups these two are the only obstructions toward being a domain, namely:
	
	\begin{conjecture}\label{first_conj} Every irreducible Artin group modulo its center is a domain.
\end{conjecture}
	
	Notice that it is conjectured that every irreducible Artin group that is not spherical (i.e., the corresponding Coxeter group is not finite) has a trivial center. So the conjecture above could be read as: for every irreducible Artin group $A$, either $A$ is not spherical, in which case it is a domain, or $A$ is spherical and $A/Z(A)$ is a domain.
	

\medskip

	In order to provide evidence for our Conjecture~\ref{first_conj}, we prove a result of independent interest, in the context of acylindrically hyperbolic groups:

	\begin{theorem}\label{main_th1} Let $G$ be an acylindrically hyperbolic group without a non-trivial finite normal subgroup. Then $G$ is a domain.
\end{theorem}

	Theorem ~\ref{main_th1} is relevant for our purposes as several classes of Artin groups are known to be acylindrically hyperbolic. In fact, relying on \cite{alg_geom_over_groups3}, we can deduce:

	\begin{corollary} If a class $\mathcal{C}$ of Artin groups is such that its irreducible members are torsion-free and acylindrically hyperbolic, then the model theory of Artin groups of type $\mathcal{C}$ reduces to the model theory of their irreducible components, i.e., the conclusions of Fact~\ref{reduction_to_components_fact} apply. In particular, this applies to Artin groups of FC type all of whose irreducible factors are non-spherical, to two-dimensional Artin groups all of whose irreducible factors are non-spherical, and to affine Artin groups (see Corollary~\ref{affine_dom}).
\end{corollary}

    This also allows us to treat the spherical case, although in this case we have to mod out the center (which is infinite cyclic for irreducible spherical Artin groups).

    \begin{corollary}\label{spherical_domain} Irreducible spherical Artin groups modulo their center are domains.
    \end{corollary}

	Although a full reduction of the model theory of $A$ to the model theory of its irreducible components requires the assumption of triviality of the center of the components $A_i$, in certain instances, Theorem~\ref{main_th1} and Fact~\ref{reduction_to_components_fact} can actually be used to deal also with cases in which the center is non-trivial (i.e., the spherical case). Most notably, this will be used in Section~\ref{intro_sub_spherical} to show the first-order rigidity of spherical Artin groups among spherical Artin groups.

\subsection{Superstability}

	The second line of investigation of this paper is on superstability. The property of superstability is one of the main ``dividing lines'' in model theory, with very strong structural consequences, e.g., its negation implies the maximal number of models in every uncountable cardinality. Many abelian groups are superstable \cite{rogers} (e.g., the finitely generated free ones), while for example non-abelian free groups are not superstable. Thus, it is a fundamental problem to understand which Artin groups are superstable. In this respect, we conjecture \mbox{that only the obvious ones are, namely:}

	\begin{conjecture} An Artin group is superstable if and only if it is abelian.
\end{conjecture}

	To the best of our knowledge, there are two main techniques to show the non-superstability of a finitely generated group $G$. The first one is to show that $G$ is residually finite and not soluble-by-finite  \cite{houcine}. The second one is to show that $G$ contains a non-degenerate hyperbolically embedded subgroup \cite[Theorem~8.1]{acy_hyper} (notice that the latter property is implied by acylindrical hyperbolicity). For various classes of Artin groups there are fundamental results concerning these properties (see e.g. \cite{Vas22, Hae22, CMMW, BGMPP, Jank22}) but  establishing either one of these two properties for all (or even ``most'') Artin groups is considered to be \mbox{out of reach at the present time.}
	
	\medskip In this paper, we follow another route, more specific to the case at hand. Specifically, we prove a technical adaptation of the non-superstability criterion isolated by the second named author et al. in\ \cite{MPS22} and we use it to reduce the problem of non-superstability to a problem on roots.

	Every non-abelian Artin group $A$ has two standard generators whose label is different from $2$, so a subgroup $H$ as in the following theorem always exists. By Lemma~\ref{dihedral_artin}, its central quotient contains a normal non-abelian free subgroup of finite index $d$; the remaining task is therefore to establish $n$-purity of $H$ in $A$ for some prime $n>d$.

\begin{theorem}\label{main_th2}
Let $(A,S)$ be a non-abelian Artin system, let $\{a,b\}\subseteq S$ have label in $\{3,4,\ldots\}\cup\{\infty\}$, and put $H=\langle a,b\rangle_A$. Choose a normal non-abelian free subgroup of $H/Z(H)$ of finite index $d$. If $H$ is $n$-pure in $A$ for some prime $n>d$, then $A$ is not superstable.
\end{theorem}

An Artin group $A(\Gamma)$ is said to be of FC type if, for every induced subgraph $\Gamma'$ of $\Gamma$ with no edges labeled by $\infty$, the Artin group $A(\Gamma')$ is spherical.
An Artin group $A(\Gamma)$ is said to be of even type if all edges of $\Gamma$ have an even label (or $\infty$).
Additionally, an Artin group $A(\Gamma)$ is said to be of large type if all edges of $\Gamma$ have a label $\geq 3$ (or $\infty$). The previous theorem, together with the results of \cite{CGGW}, \cite{CMV} and \cite{AntoFon} respectively, yields the following:

\begin{corollary}\label{cor_super}
    Let $A$ be a non-abelian Artin group of one of the following types: spherical, large, or even FC type. Then $A$ is not superstable.
\end{corollary}


	The large-type case is already covered by known acylindrical-hyperbolicity results. Indeed, a large-type Artin group with at least three generators and at least one finite label is irreducible and two-dimensional, so it is acylindrically hyperbolic by \cite[Theorem~A]{Vas22}. If all labels are infinite, it is a free group; the remaining non-abelian rank-two cases have non-elementary virtually free central quotients by Lemma~\ref{dihedral_artin}. Thus \cite[Theorem~8.1]{acy_hyper}, together with interpretability of the central quotient, also gives non-superstability in these cases. Our criterion provides a uniform approach through root closure of parabolic subgroups, which may be useful beyond the classes currently covered by acylindrical hyperbolicity.

\subsection{Spherical Artin groups}\label{intro_sub_spherical}

	The third line of investigation of the present paper is on elementary equivalence. This is a classical theme in model theory, which inspired very beautiful and deep results, such as the solution to Tarski's problem for free groups on the elementary equivalence of non-abelian free groups of different rank \cite{KM06, Sela06}. Of course, free groups are Artin groups, and in fact they are right-angled Artin groups (RAAGs), i.e., all edges are labeled by $2$ or $\infty$.
    For other results on the problem of elementary equivalence of RAAGs, see e.g. \cite{CRKR10, CRK11, CRKNG}. As promised at the beginning of the introduction, we focus here on Artin groups which are ``far from RAAGs''.

	One of the most studied classes of Artin groups is definitely the class of spherical Artin groups.
    Partial results on the problem of elementary equivalence can be found in \cite{KPV18}, where it is shown that for any of the three infinite families of irreducible spherical Artin groups, any two non-isomorphic groups in the family are not elementarily equivalent. We extend this result to the whole class of spherical Artin groups (not necessarily irreducible), proving the following rigidity result:

	\begin{theorem}\label{main_th3} Let $A$ and $B$ be two spherical Artin groups. Then $A$ is elementarily equivalent to $B$ if and only if $A$ is isomorphic to $B$.
    In other words, spherical Artin groups are first-order rigid in the class of spherical Artin groups.
\end{theorem}

	Theorem~\ref{main_th3} relies among other things on fundamental hyperbolic results of S.~Andr\'e from \cite{andre_hyper}. Also, we stress once again that the use of the techniques from Section~\ref{irr_component_intro} was crucial here, in the passage from irreducible to reducible spherical Artin groups, thus showing that those techniques also apply (with some restrictions) to torsion-free Artin groups, as already mentioned there.
	
	Finally, we observe that the isomorphism problem for Artin groups of spherical type was solved only in 2004 \cite{Par04}; therefore, comparing Artin groups is not an easy task, even in the spherical case.

    We end our investigations on the model theory of spherical Artin groups by asking the following question.

    \begin{question} Let $A$ be a spherical Artin group. Is $A$ first-order rigid? That is, is it the case that if $B$ is finitely generated and $B$ is elementarily equivalent to $A$, then $B$ is isomorphic to $A$?
      
    \end{question}
	
\subsection{Affine Artin groups}\label{intro_sub_affine}

    The next most important class of Artin groups ``far from RAAGs'' is arguably the class of affine Artin groups, i.e., those for which the corresponding Coxeter group is affine.
    Many crucial open questions on affine Artin groups were solved in recent years, such as the word problem, the triviality of the center \cite{McCammondSulway17}, and the $K(\pi, 1)$ conjecture \cite{P-Salvetti21}.

\smallskip
    
    We pose the following fundamental question:
    
	\begin{question}\label{affine_question} Let $A$ and $B$ be affine Artin groups. Is it the case that $A$ and $B$ are elementarily equivalent if and only if they are isomorphic?
\end{question}

    We were unable to give a complete answer to Question~\ref{affine_question} and we believe that a solution to this problem might be challenging and require new non-trivial algebraic results. On the other hand, we were able to show the following; recall that an Artin group $A(\Gamma)$ is called {\em simply laced} when the associated Coxeter matrix $m$ is such that $m_{s,t} \in \{2, 3\}$ for all generators $s \neq t \in \Gamma$.  Before stating the theorem, note that irreducible affine Artin groups are known to be acylindrically hyperbolic and torsion-free (cf.~\cite{calvez22} and \cite{McCammondSulway17}) and so the model theory of affine Artin groups reduces to the model theory of their irreducible components by Section~\ref{irr_component_intro}. Consequently, we focus on irreducible affine Artin groups in what follows.

    \begin{theorem}\label{affArt_existeq} Let $A$ be an Artin group of type $\tilde{A}_n$, with $n \geq 4$. Then for every irreducible simply laced affine Artin group $B$, we have that $A$ and $B$ are $\exists$-elementarily equivalent (i.e., they satisfy the same existential sentences) if and only if they are isomorphic.
    \end{theorem} 

    Our proof uses the endomorphism classification in \cite{tilde_An} and known computations of second integral homology. We recall below a simply laced specialization of the Clancy--Ellis formula, as stated in \cite[Theorem~2.6]{AkitaLiu18}, and give a direct calculation in Section~6.

    \begin{theorem}
    	Let $\Gamma$ be a simply laced irreducible Coxeter graph, and suppose that the $K(\pi, 1)$ conjecture holds for the Artin group $A(\Gamma)$. Then $H_2(A(\Gamma); \mathbb{Z}) \cong \mathbb{Z}^b \oplus (\mathbb{Z} / 2\mathbb{Z})^c$, where $b \in \mathbb{N}$ is the first Betti number of the graph $\Gamma$ and $c \in \mathbb{N}$ is another natural number which can be explicitly computed from $\Gamma$.
    In particular, if $\Gamma$ is a tree, then $H_2(A(\Gamma); \mathbb{Z}) \cong (\mathbb{Z} / 2\mathbb{Z})^c$.
\end{theorem}

    The integral homology values for types $\tilde D_n$ and $\tilde E_6,\tilde E_7,\tilde E_8$ were obtained in \cite[Corollary~3.21 and Example~3.24]{AkitaLiu18} without the affine $K(\pi,1)$ conjecture. We retain a direct calculation for completeness and use these values as obstructions to retractions. The solution of the affine $K(\pi,1)$ conjecture in \cite{P-Salvetti21} also permits a uniform calculation from the Salvetti complex.

    To the best of our knowledge, a description of $\mrm{End}(A)$ is not known for affine Artin groups except $\tilde{A}_n$ with $n \geq 4$, and so we are not currently able to extend our proof strategy beyond this case.

\section{Preliminaries}\label{prelim}

Artin groups, also known as Artin-Tits groups, are one of the most famous families of groups in geometric group theory, together with their relatives, Coxeter groups. While many results are known about Coxeter groups, the same cannot be said for Artin groups, whose properties remain in general rather obscure. Some examples of Artin groups are free groups, free abelian groups and braid groups.

Both Coxeter and Artin groups are defined by a presentation, encoded by a matrix or, equivalently, by a graph. Let $S$ be a finite set and $m$ a symmetric matrix whose rows and columns are both in bijective correspondence with $S$. Suppose that for $s,t\in S$ the element $m_{s,t}$ of $m$ satisfies $m_{s,t}=1$ if $s=t$ and $m_{s,t}\in\{2,\dots,\infty \}$ if $s\ne t$. Such a matrix is called a Coxeter matrix over $S$.

Given a Coxeter matrix $m$, we can define two groups associated to $S$ and $m$, the Coxeter group $W$ and the Artin group $A$, given by the following presentations:
\[ W=\langle S \mid s^2=e,\; \underbrace{st\dots}_{m_{s,t}}=\underbrace{ts\dots}_{m_{t,s}}\,\forall s\ne t \text{ s.t.~} m_{s,t}\ne\infty\rangle\]
\[A=\langle S \mid \underbrace{st\dots}_{m_{s,t}}=\underbrace{ts\dots}_{m_{t,s}}\,\forall s\ne t \text{ s.t.~} m_{s,t}\ne\infty\rangle\]

The data encoded by $m$ can also be represented by means of a graph, known as the Coxeter graph of $A$ and $W$. The vertices of this graph are the elements of $S$, and two vertices $s,t$ are connected by an edge if and only if $m_{s,t}\geq 3$. The edge connecting $s$ and $t$ is labeled $m_{s,t}$ if $m_{s,t}\geq 4$.
Therefore, two non-connected vertices represent commuting generators. Note that while this convention is the standard one for Coxeter groups, another convention is often used for Artin groups, namely that two vertices $s$ and $t$ in the graph are connected by an edge if and only if $m_{s,t}\ne\infty$, and such an edge is labeled $m_{s,t}$ if $m_{s,t} \geq 3$. However, we will employ the first convention, as it expresses more valuable information for our purposes.
The Artin (resp.~Coxeter) group associated to the Coxeter graph $\Gamma$ is denoted $A(\Gamma)$ (resp. $W(\Gamma)$).

Next, we define key families of Artin groups for our main arguments, as well as ancillary ones needed to frame the discussion.

\begin{definition}
	An Artin group $A(\Gamma)$ is said to be spherical or of finite type if the corresponding Coxeter group $W(\Gamma)$ is finite; otherwise, it is said to be non-spherical or of infinite type.
\end{definition}
\begin{definition}
	An Artin group $A(\Gamma)$ is said to be right-angled (RAAG for short) if $m_{s,t}\in\{2,\infty\}$ for all distinct $s,t \in S$, i.e., if the only relations that we see in the presentation are commutators among generators.
\end{definition}
\begin{definition}
    An Artin group $A(\Gamma)$ is said to be affine (or Euclidean) if the corresponding Coxeter group is an affine Coxeter group.
\end{definition}

Let $\Gamma$ be a Coxeter graph and $X$ be a subset of the set of vertices of $\Gamma$. Then the subgraph $\Gamma'$ spanned by $X$ is a Coxeter graph and $\langle X \rangle \leq A(\Gamma)$ is isomorphic to the Artin group $A(\Gamma')$ defined by the subgraph $\Gamma'$ (see \cite{vdL83}). Such a subgroup is called a standard parabolic subgroup of $A(\Gamma)$.

\begin{definition}
	An Artin group $A(\Gamma)$ is said to be an even Artin group of FC type if all the labels in $\Gamma$ are even (or $\infty$) and, for any induced subgraph $\Gamma'$ of $\Gamma$ with no edges labeled $\infty$, the corresponding standard parabolic subgroup $A(\Gamma')$ is a spherical Artin group.
\end{definition}

Given the definition of a standard parabolic subgroup, the convention for Coxeter graphs that we use in this paper yields that an Artin group $A(\Gamma)$ coincides with the direct product $A(\Gamma_1)\times\cdots\times A(\Gamma_n)$, where the $\Gamma_i$'s are the connected components of $\Gamma$. The Artin groups corresponding to connected Coxeter graphs are called irreducible. There is a complete classification of spherical Artin groups, whose irreducible factors are those listed in Figure \ref{irredCoxdiagr} (the figure was taken from \cite{Hum90}). Therefore, an Artin group is spherical if, and only if, its Coxeter graph has the graphs of Figure~\ref{irredCoxdiagr} as connected components.

\begin{figure}
	\centering
	\begin{tikzpicture}
		\node () at (0,9) {$A_n$\;$(n\geq 1):$};
		\draw[thick] (2,9) -- (4,9)
					(5,9) -- (6,9);
		\draw[dashed, thick] (4,9) -- (5,9);
		\node[circle, draw, fill, text width=1mm, inner sep=0.5] at (2,9) {};
		\node[circle, draw, fill, text width=1mm, inner sep=0.5] at (3,9) {};
		\node[circle, draw, fill, text width=1mm, inner sep=0.5] at (4,9) {};
		\node[circle, draw, fill, text width=1mm, inner sep=0.5] at (5,9) {};
		\node[circle, draw, fill, text width=1mm, inner sep=0.5] at (6,9) {};
		
		\node[below, font=\small] at (2,9) {$s_1$};
		\node[below, font=\small] at (3,9) {$s_2$};
		\node[below, font=\small] at (4,9) {$s_3$};
		\node[below, font=\small] at (5,9) {$s_{n-1}$};
		\node[below, font=\small] at (6,9) {$s_n$};

		\node () at (0,8) {$B_n$\;$(n\geq 2):$};
		\draw[thick] (2,8) -- (4,8)
		(5,8) -- (6,8);
		\draw[dashed, thick] (4,8) -- (5,8);
		\node[circle, draw, fill, text width=1mm, inner sep=0.5] at (2,8) {};
		\node[circle, draw, fill, text width=1mm, inner sep=0.5] at (3,8) {};
		\node[circle, draw, fill, text width=1mm, inner sep=0.5] at (4,8) {};
		\node[circle, draw, fill, text width=1mm, inner sep=0.5] at (5,8) {};
		\node[circle, draw, fill, text width=1mm, inner sep=0.5] at (6,8) {};
		
		\node[below, font=\small] at (2,8) {$s_1$};
		\node[below, font=\small] at (3,8) {$s_2$};
		\node[below, font=\small] at (4,8) {$s_3$};
		\node[below, font=\small] at (5,8) {$s_{n-1}$};
		\node[below, font=\small] at (6,8) {$s_n$};
		\node[above, font=\small] at (2.5,8) {4};

		\node () at (0,7) {$D_n$\;$(n\geq 4):$};
		\draw[thick] (2,7.3) -- (3,7) -- (4,7)
		(2,6.7) -- (3,7)
		(5,7) -- (6,7);
		\draw[dashed, thick] (4,7) -- (5,7);
		\node[circle, draw, fill, text width=1mm, inner sep=0.5] at (2,7.3) {};
		\node[circle, draw, fill, text width=1mm, inner sep=0.5] at (2,6.7) {};
		\node[circle, draw, fill, text width=1mm, inner sep=0.5] at (3,7) {};
		\node[circle, draw, fill, text width=1mm, inner sep=0.5] at (4,7) {};
		\node[circle, draw, fill, text width=1mm, inner sep=0.5] at (5,7) {};
		\node[circle, draw, fill, text width=1mm, inner sep=0.5] at (6,7) {};
		
		\node[below, font=\small] at (2,7.3) {$s_1$};
		\node[below, font=\small] at (2,6.7) {$s_2$};
		\node[below, font=\small] at (3,7) {$s_3$};
		\node[below, font=\small] at (4,7) {$s_4$};
		\node[below, font=\small] at (5,7) {$s_{n-1}$};
		\node[below, font=\small] at (6,7) {$s_n$};

		\node () at (0,6) {$E_6:$};
		\draw[thick] (2,6) -- (6,6)
		(4,6) -- (4,5.5);
		
		\node[circle, draw, fill, text width=1mm, inner sep=0.5] at (2,6) {};
		\node[circle, draw, fill, text width=1mm, inner sep=0.5] at (3,6) {};
		\node[circle, draw, fill, text width=1mm, inner sep=0.5] at (4,6) {};
		\node[circle, draw, fill, text width=1mm, inner sep=0.5] at (5,6) {};
		\node[circle, draw, fill, text width=1mm, inner sep=0.5] at (6,6) {};
		\node[circle, draw, fill, text width=1mm, inner sep=0.5] at (4,5.5) {};
		
		\node[below, font=\small] at (2,6) {$s_1$};
		\node[right, font=\small] at (4,5.5) {$s_2$};
		\node[below, font=\small] at (3,6) {$s_3$};
		\node[below left, font=\small] at (4,6) {$s_4$};
		\node[below, font=\small] at (5,6) {$s_5$};
		\node[below, font=\small] at (6,6) {$s_6$};

		\node () at (0,5) {$E_7:$};
		\draw[thick] (2,5) -- (7,5)
		(4,5) -- (4,4.5);
		
		\node[circle, draw, fill, text width=1mm, inner sep=0.5] at (2,5) {};
		\node[circle, draw, fill, text width=1mm, inner sep=0.5] at (3,5) {};
		\node[circle, draw, fill, text width=1mm, inner sep=0.5] at (4,5) {};
		\node[circle, draw, fill, text width=1mm, inner sep=0.5] at (5,5) {};
		\node[circle, draw, fill, text width=1mm, inner sep=0.5] at (6,5) {};
		\node[circle, draw, fill, text width=1mm, inner sep=0.5] at (7,5) {};
		\node[circle, draw, fill, text width=1mm, inner sep=0.5] at (4,4.5) {};
		
		\node[below, font=\small] at (2,5) {$s_1$};
		\node[right, font=\small] at (4,4.5) {$s_2$};
		\node[below, font=\small] at (3,5) {$s_3$};
		\node[below left, font=\small] at (4,5) {$s_4$};
		\node[below, font=\small] at (5,5) {$s_5$};
		\node[below, font=\small] at (6,5) {$s_6$};
		\node[below, font=\small] at (7,5) {$s_7$};

		\node () at (0,4) {$E_8:$};
		\draw[thick] (2,4) -- (8,4)
		(4,4) -- (4,3.5);
		
		\node[circle, draw, fill, text width=1mm, inner sep=0.5] at (2,4) {};
		\node[circle, draw, fill, text width=1mm, inner sep=0.5] at (3,4) {};
		\node[circle, draw, fill, text width=1mm, inner sep=0.5] at (4,4) {};
		\node[circle, draw, fill, text width=1mm, inner sep=0.5] at (5,4) {};
		\node[circle, draw, fill, text width=1mm, inner sep=0.5] at (6,4) {};
		\node[circle, draw, fill, text width=1mm, inner sep=0.5] at (7,4) {};
		\node[circle, draw, fill, text width=1mm, inner sep=0.5] at (8,4) {};
		\node[circle, draw, fill, text width=1mm, inner sep=0.5] at (4,3.5) {};
		
		\node[below, font=\small] at (2,4) {$s_1$};
		\node[right, font=\small] at (4,3.5) {$s_2$};
		\node[below, font=\small] at (3,4) {$s_3$};
		\node[below left, font=\small] at (4,4) {$s_4$};
		\node[below, font=\small] at (5,4) {$s_5$};
		\node[below, font=\small] at (6,4) {$s_6$};
		\node[below, font=\small] at (7,4) {$s_7$};
		\node[below, font=\small] at (8,4) {$s_8$};

		\node () at (0,3) {$F_4:$};
		\draw[thick] (2,3) -- (5,3);
		
		\node[circle, draw, fill, text width=1mm, inner sep=0.5] at (2,3) {};
		\node[circle, draw, fill, text width=1mm, inner sep=0.5] at (3,3) {};
		\node[circle, draw, fill, text width=1mm, inner sep=0.5] at (4,3) {};
		\node[circle, draw, fill, text width=1mm, inner sep=0.5] at (5,3) {};
				
		\node[below, font=\small] at (2,3) {$s_1$};
		\node[below, font=\small] at (3,3) {$s_2$};
		\node[below, font=\small] at (4,3) {$s_3$};
		\node[below, font=\small] at (5,3) {$s_4$};
		\node[above, font=\small] at (3.5,3) {4};

		\node () at (0,2) {$H_3:$};
		\draw[thick] (2,2) -- (4,2);
		
		\node[circle, draw, fill, text width=1mm, inner sep=0.5] at (2,2) {};
		\node[circle, draw, fill, text width=1mm, inner sep=0.5] at (3,2) {};
		\node[circle, draw, fill, text width=1mm, inner sep=0.5] at (4,2) {};
		
		\node[below, font=\small] at (2,2) {$s_1$};
		\node[below, font=\small] at (3,2) {$s_2$};
		\node[below, font=\small] at (4,2) {$s_3$};
		\node[above, font=\small] at (2.5,2) {5};

		\node () at (0,1) {$H_4:$};
		\draw[thick] (2,1) -- (5,1);
		
		\node[circle, draw, fill, text width=1mm, inner sep=0.5] at (2,1) {};
		\node[circle, draw, fill, text width=1mm, inner sep=0.5] at (3,1) {};
		\node[circle, draw, fill, text width=1mm, inner sep=0.5] at (4,1) {};
		\node[circle, draw, fill, text width=1mm, inner sep=0.5] at (5,1) {};
		
		\node[below, font=\small] at (2,1) {$s_1$};
		\node[below, font=\small] at (3,1) {$s_2$};
		\node[below, font=\small] at (4,1) {$s_3$};
		\node[below, font=\small] at (5,1) {$s_4$};
		\node[above, font=\small] at (2.5,1) {5};

		\node () at (0,0) {$I_2(m)$\;$(m\geq 5, m\ne\infty):$};
		\draw[thick] (4,0) -- (5,0);
		
		\node[circle, draw, fill, text width=1mm, inner sep=0.5] at (4,0) {};
		\node[circle, draw, fill, text width=1mm, inner sep=0.5] at (5,0) {};
		
		\node[below, font=\small] at (4,0) {$s_1$};
		\node[below, font=\small] at (5,0) {$s_2$};
		\node[above, font=\small] at (4.5,0) {$m$};
	\end{tikzpicture}
	\caption{The Coxeter graphs of the irreducible spherical Artin groups}
	\label{irredCoxdiagr}
\end{figure}
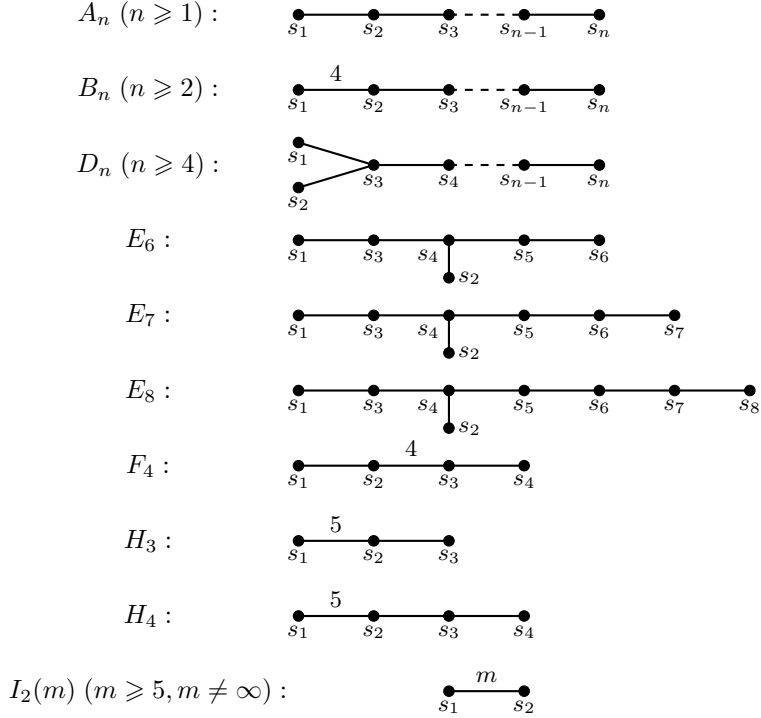

Spherical Artin groups are an example of Garside groups, an important class of groups with rich combinatorics. These are groups of fractions of some particular monoids, called Garside monoids, that contain a distinguished element $\Delta$ playing an important role in such combinatorics. We will not recall the details and results of Garside theory; instead, we refer to \cite{Dehor02, found_gars_th, DehorPar}. For the sake of clarity, we specify that whenever we cite the Garside structure of a spherical Artin group we refer to the standard one. The only result we need is the following:

\begin{theorem}[\cite{BriSa}]\label{censphAr}
	Let $A$ be an irreducible spherical Artin group. Then $Z(A)$ is infinite cyclic and generated by either $\Delta$ or $\Delta^2$, where $\Delta$ is the standard Garside element of $A$. In particular, the generator is $\Delta^2$ if the Coxeter graph associated to $A$ is of type $A_n$ with $n\geq 2$, $D_n$ with $n$ odd, $E_6$ and $I_2(m)$ with $m$ odd, and $\Delta$ in all the other cases.
\end{theorem}

\begin{remark}\label{fundam_elem}
	An explicit description of the generator $\Delta$ or $\Delta^2$ of the center can be found in \cite[Theorem 3.3]{Par04} and in \cite{BriSa}, together with \cite{Hum90}.
\end{remark}

In Section~\ref{sec_elem_eq_spher} we will need the following propositions, the first of which is a restricted version of \cite[Cor. 8.3]{CGGW}.

\begin{proposition}[\cite{CGGW}]\label{rootparsgb}
	Let $A$ be a spherical Artin group and $P$ be a standard parabolic subgroup of $A$. If $g\in A$ is such that $g^k\in P$ for some nonzero $k$, then $g\in P$.
\end{proposition}

\begin{proposition} [\cite{Par04}]\label{delparsgb}
	Let $A$ be a spherical Artin group and $P$ be a proper standard parabolic subgroup of $A$. Let $\Delta$ be the standard Garside element of $A$. Then $P\cap \langle\Delta\rangle=\{e\}$.
\end{proposition}

The following fact will also be needed in Section \ref{sec_elem_eq_spher}.

\begin{fact}\label{center_dec_fact} Let $(A, S)$ be a spherical Artin system and let $A = A_1 \times \cdots \times A_n$ be the corresponding decomposition into irreducible factors. Then we have:
	$$A/Z(A) \cong A_1/Z(A_1) \times \cdots \times A_n/Z(A_n).$$
	This is a consequence of the First Isomorphism Theorem for groups, applied to the homomorphism $A\to A_1/Z(A_1)\times\cdots\times A_n/ Z(A_n)$. Note that $Z(A)=Z(A_1)\times\cdots\times Z(A_n)$.
\end{fact}

It is known that spherical Artin groups are torsion-free (this is a consequence of the $K(\pi,1)$ conjecture holding for spherical Artin groups; see for instance \cite{ParConj}). However, taking the quotient by the center produces torsion elements, whose possible orders in the irreducible case are listed in \cite{Sor21}. We list the possible orders in Table \ref{ord_tors_elem_table}, as this will be relevant for the proof of Theorem~\ref{main_th3} in Section~\ref{sec_elem_eq_spher}.

\begin{table}
	\[
\begin{array}{cc}
\toprule
\text{Coxeter graph} & \text{orders of torsion elements} \\
\toprule
A_n,\,n\geq 2 & \text{all divisors of } n, n+1 \\
\midrule
B_n,\,n\geq 2 & \text{all divisors of } n \\
\midrule
D_n,\,n\text{ even}\geq 4 & \text{all divisors of } n-1, n/2 \\
\midrule
D_n,\,n\text{ odd}\geq 5 & \text{all divisors of } 2n-2, n \\
\midrule
E_6 & 2,3,4,6,8,9,12 \\
\midrule
E_7 & 3,7,9 \\
\midrule
E_8 & 2,3,4,5,6,10,12,15 \\
\midrule
F_4 & 2,3,4,6 \\
\midrule
H_3 & 3,5 \\
\midrule
H_4 & 2,3,5,6,10,15 \\
\midrule
I_2(m),\,m \text{ even}\geq 6 & \text{all divisors of } m/2 \\
\midrule
I_2(m),\,m \text{ odd}\geq 5 & \text{all divisors of } 2, m\\
\bottomrule
\end{array}
\]
\caption{Orders of torsion elements of irreducible spherical Artin groups modulo their center}
\label{ord_tors_elem_table}
\end{table}

\section{When is an Artin group a domain?}

	\begin{definition} We say that $A$ is a domain if for every $x, y \in A$, whenever $x, y \neq e$ there is $g \in A$ such that $[x, y^g] \neq e$.
\end{definition}

We now recall and prove \ref{main_th1}. 

\begin{theorem1.3}
    Let $G$ be an acylindrically hyperbolic group without a non-trivial normal finite subgroup. Then $G$ is a domain.
\end{theorem1.3}

	\begin{proof}
    Suppose, for a contradiction, that $G$ is not a domain. Then it satisfies the following $\exists\forall$-sentence:
    \[
    \phi: \exists x\ne1\,\,\exists y\ne1\,\,\forall g\,\, [x,gyg^{-1}]=1.
    \]
    By \cite{andre_fruch}, $G$ has the same $\exists\forall$-theory as the free product $G*F_2$, hence $G*F_2\models  \phi$. But one can see that this is not possible, taking as $g$ a suitable element in the $F_2$ free factor.

\smallskip \noindent
Let us give an alternative proof relying on less sophisticated arguments. By \cite[Lemma 5.6]{hull}, there exist three loxodromic elements $g_1,g_2,g_3\in G$ such that, for every $i,j\in\lbrace 1,2,3\rbrace$ with $i\neq j$, the following conditions hold:
\begin{enumerate}[(1)]
    \item the maximal virtually cyclic subgroup $M(g_i)$ of $G$ containing $g_i$ is equal to $\langle g_i\rangle$;
    \item $\langle g_i\rangle \cap \langle g_j\rangle$ is trivial.
\end{enumerate}
Let $x,y$ be non-trivial elements of $G$. Pick $g_i$ with $i\in\lbrace 1,2,3\rbrace$ such that $x\notin M(g_i)$ and $y\notin M(g_i)$ (this is possible by the conditions above). By \cite[Lemma 3.3]{abbott_dahmani}, the element $[x,g_i^nyg_i^{-n}]=xg_i^nyg_i^{-n}x^{-1}g_i^ny^{-1}g_i^{-n}$ is non-trivial for $n\geq 1$ sufficiently large, and thus $G$ is a domain.
	\end{proof}

\begin{corollary1.5}
Irreducible spherical Artin groups modulo their center are domains.
\end{corollary1.5}
\begin{proof}
Let $A$ be an irreducible spherical Artin group. If $A$ has type $A_1$, then $A/Z(A)$ is trivial and satisfies the definition of a domain vacuously. Hence assume that $A$ is non-abelian. If the center of $A$ is generated by $\Delta^2$, then $A/Z(A)$ coincides with the group $\mathcal{G}$ of \cite{npc_art_fin} and by \cite[Theorem 4.13]{npc_art_fin} it does not contain any nontrivial finite normal subgroup. On the other hand, if $\Delta$ is central in $A$, then by the Third Isomorphism Theorem $A/Z(A)\simeq\mathcal{G}/C_2$, where $\mathcal{G}=A/\langle\Delta^2\rangle$ as in \cite{npc_art_fin} and $C_2 = \langle\Delta\rangle/\langle\Delta^2\rangle$. Therefore, any finite normal subgroup of $A/Z(A)$ lifts to a finite normal subgroup of $\mathcal{G}$. Applying again \cite[Theorem 4.13]{npc_art_fin}, we obtain that $A/Z(A)$ does not contain any nontrivial finite normal subgroup. In both cases, $A/Z(A)$ is acylindrically hyperbolic by \cite{CalvezWiest17} and has no nontrivial finite normal subgroups. Thus, by Theorem~\ref{main_th1}, it is a domain. 
\end{proof}

\begin{corollary}\label{affine_dom}
    Irreducible affine Artin groups are domains.
\end{corollary}
\begin{proof}
    By \cite{calvez22} and \cite{McCammondSulway17}, irreducible affine Artin groups are acylindrically hyperbolic and torsion-free. By \ref{main_th1}, then, they are domains.
\end{proof}

\section{Non-superstability of non-abelian Artin groups}

	The core of this section is an adaptation of the criterion for non-superstability from \cite{MPS22} which will allow us to prove Theorem~\ref{main_th2}. This criterion (namely \ref{general_criterion}) is of independent interest and could be applied in other contexts.

\begin{notation}\label{notation_strings} Given $(y_i : i < \omega)$ and $n < \omega$, let $\bar{y}_{[n)}= (y_i : i < n)$. Also, given $\{i_0, \dots, i_{k-1} \} = I \subseteq \{ 0, \dots, n-1 \}$ we let $\bar{y}_{I} = (y_{i_\ell} : \ell < k)$. The two notations agree when we write $\{0, \dots, n-1 \} = [0, n) = [n)$.
\end{notation}

	\begin{notation} Given $\eta, \nu \in \omega^{<\omega}$, we write $\nu \triangleleft\eta$ if $\eta$ properly extends $\nu$, and $\nu \trianglelefteq \eta$ if $\eta$ extends $\nu$ or $\eta = \nu$. Also, we identify the number $n < \omega$ with the set $\{0, \dots, n-1\}$.
\end{notation}

	\begin{definition}\label{def_unsuper} We say that the first-order theory $T$ is {\em not} superstable if:
\begin{enumerate}[(a)]
	\item\label{a} there are formulas $\varphi_n(x, \bar{y}_{[k_n)})$, for $n < \omega$ and $k_n = k(n) \geq n$;
	\item\label{b} there is $M \models T$;
	\item\label{c} there are $b_{\eta} \in M$, for $\eta \in \omega^{<\omega}$;
	\item\label{d} there are $\bar{a}_\nu \in M^{k(n)}$, for $\nu \in \omega^n$;
	\item\label{e} for $\nu \in \omega^n$ and $\eta \in \omega^{<\omega}$, $M \models \varphi_n(b_{\eta}, \bar{a}_\nu)$, if $\nu \triangleleft \eta$;
	\item\label{f} there is $m(n) < \omega$ such that if $\nu \in \omega^n$, $k, j < \omega$, and $\eta = \nu^\frown (k, j)$, then:
	$$|\{i < \omega: M \models \varphi_{n+1}(b_\eta, \bar{a}_{\nu^{\frown} (i)}) \} | \leq m(n).$$
\end{enumerate}
\end{definition}

	There are many equivalent definitions of superstability, and we use the above for convenience. In the countable languages considered here, a complete theory $T$ is superstable if it is $\kappa$-stable for every $\kappa\geq 2^{\aleph_0}$. Here $\kappa$-stability means that, for every $M\models T$, every $A\subseteq M$ of cardinality at most $\kappa$, and every positive integer $r$, the set of complete $r$-types over $A$ has cardinality at most $\kappa$; see \cite[Chapter~4]{marker}. An important example is the abelian group $\mathbb Z$.

	\begin{remark}\label{rmk_superst} {\rm In Definition~\ref{def_unsuper}(\ref{f}) we can take $m(n) = 1$, for every $n < \omega$. Indeed, without loss of generality $M$ is $\aleph_1$-saturated, and so it is enough to find $(\bar{a}_\eta, b_\eta : \eta \in \omega^{ \leq n+2})$, for every $n < \omega$. To this end, let $\nu \in \omega^{n}$ and consider the function $f_\nu: \omega^3 \rightarrow \{ 0, 1\}$ such that $f_\nu(k, j, i) = 1$, if $k = i$ or $M \not \models \varphi_{n+1}(b_{\nu^{\frown} (k, j)}, \bar{a}_{\nu^{\frown} (i)})$, and $f_\nu(k, j, i) = 0$ otherwise. By the Infinite Ramsey Theorem there is an infinite $f_\nu$-homogeneous subset of $\omega$, and by clause (\ref{f}) of Definition~\ref{def_unsuper} this set has to have color $1$, and so we can conclude easily.}
\end{remark}

	\begin{remark}\label{rmk_superst+} {\rm In Definition~\ref{def_unsuper}(\ref{c}-\ref{d}) we can restrict to $\eta$'s and $\nu$'s in the set:
$$\mathrm{inc}_{<\omega}(\omega) = \{ \sigma \in \omega^{<\omega} : \sigma(0) < \sigma (1) < \cdots < \sigma(|\sigma|-1) \}.$$
We denote the subset of $\mathrm{inc}_{<\omega}(\omega)$ consisting of the sequences with $\{0, \dots, m-1\}$ as domain by $\mathrm{inc}_{m}(\omega)$.}
\end{remark}

\begin{definition}\label{def_pure} Let $H \leq G$ and $2 \leq n < \omega$. We say that $H$ is $n$-pure in $G$ if, for every $a \in H$, whenever there is $b \in G$ such that $b^n = a$ we have that $b \in H$.
\end{definition}

\begin{theorem}\label{general_criterion} Let $G$ be a group. A sufficient condition for the non-superstability of $G$ is that there are subgroups $N \trianglelefteq H \leq G$, $2 \leq n < \omega$, and $1 \leq k < \omega$ such that:
	\begin{enumerate}[(a)]
	\item $H$ is $n$-pure in $G$, that is, $x\in G$ and $x^n\in H$ imply $x\in H$; moreover, for every $a\in H$, $\bigl|\{xN\in H/N:(xN)^n=aN\}\bigr|\leq k$. 
	\item writing, for a subset $Y\subseteq H$, $Y/N=\{yN:y\in Y\}$, there are $a_\ell \in H \setminus N$, for $\ell < \omega$, such that $\ell_1 \neq \ell_2$ implies that $a^{-1}_{\ell_1} a_{\ell_2}N \notin \{ x^ny^nN : x, y \in H \}$.
\end{enumerate}
\end{theorem}

	\begin{proof} Let $n$ and $k$ be as in the statement of the theorem. By induction on $m < \omega$, we define the group word $w_m(z, \bar{y}_{[m)})$ (recall Notation~\ref{notation_strings}) as follows:
	\begin{enumerate}[(i)]
	\item $w_0(z) = z$;
	\item $w_{m+1}(z, \bar{y}_{[m+1)}) = w_m(y_mz^n, \bar{y}_{[m)})$.
	\end{enumerate}
	Notice that if $m < \omega$ is such that $m = m_1 + m_2$, then:
	$$w_m(x, \bar{y}_{[m)}) = w_{m_2}(w_{m_1}(x, \bar{y}_{[m_2,m)}), \bar{y}_{[m_2)}).$$
\smallskip \noindent Let now $\varphi_m(x, \bar{y}_{[m)})$ be the formula:
	\begin{equation}\label{equation1} \exists z(x = w_m(z, \bar{y}_{[m)})),
\tag{$\star_1$}  
\end{equation}
(clearly for $m = 0$, the set $[m, m)$ is simply $\varnothing$ and so $w_0(z, y_{[m, m)}) = w_0(z) = z$.)
	Notice that:
	\begin{equation}\label{equation2}
	\varphi_{m+1}(x, \bar{y}_{[m+1)}) \vdash \varphi_{m}(x, \bar{y}_{[m)}) \vdash \cdots \tag{$\star_2$}  \end{equation}
	We claim that $(\varphi_{m}(x, \bar{y}_{[m)}) : m < \omega)$ is a witness for the non-superstability of $G$, referring here to Definition~\ref{def_unsuper} (cf. also Remarks~\ref{rmk_superst}~and~\ref{rmk_superst+}).

\smallskip
\noindent Let $a_\ell \in H \setminus N$, for $\ell < \omega$, be as in the statement of the theorem. Now, for $m < \omega$ and $\nu \in \omega^{m}$, let:
\begin{enumerate}[$(\star)_3$]
\item $\bar{a}_\nu = (a_{\nu(\ell)} : \ell < m) \in H^{m}$.
\end{enumerate}
\smallskip \noindent Set $b_{\varnothing}=e$. For $m < \omega$ and $\eta \in \omega^{m+1}$ we have:
\begin{equation}\label{equation4} b_\eta = w_m(a_{\eta(m)}, \bar{a}_{\eta \restriction m}) \in H
\tag{$\star_4$}. \end{equation}
Clearly clauses (\ref{a})-(\ref{d}) of Definition~\ref{def_unsuper} hold. Furthermore, we have:
\begin{equation}\label{equation5}
\eta \in \omega^{m+1} \;\; \Rightarrow \;\; G \models \varphi_m(b_\eta, \bar{a}_{\eta\restriction m}).
\tag{$\star_5$} \end{equation}
Notice that $(\star_5)$ is immediate by the choice of $b_{\eta}$.
\newline More strongly, we have:
\begin{equation}\label{equation6} \eta \in \omega^{<\omega}, \quad \nu \triangleleft \eta \quad \Longrightarrow \;\; G \models \varphi_{|\nu|}(b_\eta, \bar{a}_{\nu}).
\tag{$\star_6$} \end{equation}
Notice that $(\star_6)$ is by $(\star)_2$ and $(\star)_5$.
\newline Further, we have that if $b \in H$ and $\eta \in \omega^{m}$, then letting:
\begin{equation}\label{equation7} C^\eta_b := \{ c \in G : G \models b = w_{m}(c, \bar{a}_{\eta}) \},
\tag{$\star_7$} \end{equation}
we have:
\begin{equation}\label{equation8} C^\eta_b \subseteq H \; \text{ and } \; |C^\eta_b/N| \leq k^m.
\tag{$\star_8$} \end{equation}
We prove $(\star_8)$ by induction on $m=|\eta|$. The case $m=0$ is immediate. Let $m=\ell+1$ and $\nu=\eta\restriction\ell$. For $c\in C_b^\eta$, put $d_c=a_{\eta(\ell)}c^n$. Then $d_c\in C_b^\nu\subseteq H$ by induction. Hence $c^n=a_{\eta(\ell)}^{-1}d_c\in H$, so $c\in H$ by $n$-purity. The map $c\mapsto d_c$ induces a well-defined map \[C_b^\eta/N\longrightarrow C_b^\nu/N,\qquad cN\longmapsto a_{\eta(\ell)}c^nN,\] because $N\trianglelefteq H$. For a fixed $dN\in C_b^\nu/N$, its fiber is contained in \[\{cN\in H/N:(cN)^n=a_{\eta(\ell)}^{-1}dN\},\] which has at most $k$ elements by (a). Consequently $|C_b^\eta/N|\leq k|C_b^\nu/N|\leq k^{\ell+1}$, as required.
\newline Finally, if $\nu\in\omega^m$, $r,s<\omega$, and $\eta=\nu^\frown(r,s)$, then:
\begin{equation}\label{equation9}
|\{i < \omega: G \models \varphi_{m+1}(b_\eta, \bar{a}_{\nu^{\frown} (i)}) \} | \leq k^m.
\tag{$\star_9$} \end{equation}
We prove $(\star_9)$. Let $\mathcal{U}^m_\eta := \{i < \omega: G \models \varphi_{m+1}(b_\eta, \bar{a}_{\nu^{\frown} (i)}) \}$, and for every $i \in \mathcal{U}^m_\eta$, choose $c_i \in G$ such that:
$$G \models b_\eta = w_{m+1}(c_i, \bar{a}_{\nu^{\frown} (i)}).$$
Note that, for $i \in \mathcal{U}^m_\eta$, $c_i \in H$, since $b_\eta \in H$ (cf. $(\star)_4$) and $c_i \in C^{\nu^{\frown} (i)}_{b_\eta}$, and, by clause $(\star)_8$, we have that  $C^{\nu^{\frown} (i)}_{b_\eta} \subseteq H$. Further, for $i \in \mathcal{U}^m_\eta$, we have:
$$G \models b_\eta = w_{m}(a_{i}c_i^n, \bar{a}_\nu),$$
and so $a_{i}c_i^n \in C^\nu_{b_\eta}$ (recall that $\nu^{\frown} (i)(m) = i$). For the sake of contradiction, assume that $|\mathcal{U}^m_\eta| > k^m$. By $(\star)_8$ we have that
$|C^\nu_{b_\eta}/N| \leq k^m$, and so for some $i \neq j \in \mathcal{U}^m_\eta$ we have $a_ic_i^nN = a_jc_j^nN$. Thus, we have:
\begin{equation}\label{equation10}
a_j^{-1}a_iN = (c_j)^n(c^{-1}_i)^nN = (c_j)^nN(c^{-1}_i)^nN.
\tag{$\star_{10}$} \end{equation}
But then, since $c_i, c_j \in H$ (as observed above), the conclusion $(\star)_{10}$ is in contradiction with clause (b) of the statement of the theorem. Thus, $(\star)_{9}$ holds indeed.
\newline Hence, by $(\star)_6$ and $(\star)_9$, conditions (\ref{e}) and (\ref{f}) of Definition~\ref{def_unsuper} are also satisfied.
\end{proof}

	\begin{lemma}\label{intermediate_lemma}
Let $N\trianglelefteq H\leq G$, and let $P=H/N$ contain a normal non-abelian free subgroup $F$ of finite index $d$. Fix a prime $n>d$ and suppose that $H$ is $n$-pure in $G$. Suppose also that there are $a_\ell\in H\setminus N$, for $\ell<\omega$, such that $a_{\ell_1}^{-1}a_{\ell_2}N\notin\{x^ny^nN:x,y\in H\}$ whenever $\ell_1\neq\ell_2$. Then the hypotheses of Theorem~\ref{general_criterion} hold with $k=1$.
\end{lemma}

	\begin{proof}
Condition (b) is an assumption, and $n$-purity gives the first part of (a). To prove the quotient-root bound, let $u,v\in P$ satisfy $u^n=v^n$. Since $|P/F|=d$, we have $u^d,v^d\in F$. Thus $(u^d)^n=(v^d)^n$, and uniqueness of roots in the free group $F$ gives $u^d=v^d$. Choose integers $r,s$ such that $rn+sd=1$, which is possible because $n>d$ is prime. Then \[u=(u^n)^r(u^d)^s=(v^n)^r(v^d)^s=v.\] Therefore every element of $P$ has at most one $n$th root, proving (a) with $k=1$.
\end{proof}

	\begin{proposition}\label{simon_prop} Let $B$ be a free product of cyclic groups which is not abelian-by-finite and let $n \geq 2$. Then there are $a_\ell \in B\setminus\{e\}$, for $\ell < \omega$, such that $\ell_1 \neq \ell_2$ implies that $a^{-1}_{\ell_1} a_{\ell_2} \notin \{ x^n y^n : x, y \in B \}$.
\end{proposition}

	\begin{proof}
Choose a free factor $B_0$ of $B$ which is a free product of finitely many nontrivial cyclic groups and is not virtually cyclic. Such a choice is possible: either two factors are available and are not both of order two, or three factors of order two can be chosen. Otherwise $B$ would be cyclic or infinite dihedral, contrary to the hypothesis. The group $B_0$ is a non-elementary virtually free, hence hyperbolic, group. By \cite[Theorem~1]{MyasnikovNikolaev14}, applied to the proper word $w(x)=x^n$, the $n$th-power width of $B_0^n=\langle x^n:x\in B_0\rangle$ is infinite. Here properness follows from $n\geq2$, since $x^n$ is nontrivial and does not take every value in the infinite cyclic group.

\smallskip \noindent For $g\in B_0^n$, let $\lambda_n(g)$ be the least number of $n$th powers in $B_0$ whose product is $g$. Inverses of $n$th powers are again $n$th powers, so \[\lambda_n(g^{-1})=\lambda_n(g),\qquad \lambda_n(gh)\leq\lambda_n(g)+\lambda_n(h).\] Since $\lambda_n$ is unbounded, choose nonidentity elements $a_0,a_1,\ldots\in B_0^n$ recursively so that \[\lambda_n(a_j)>2+\max_{i<j}\lambda_n(a_i)\qquad(j\geq1).\] If $i<j$ and $a_i^{-1}a_j=x^ny^n$ with $x,y\in B_0$, then $\lambda_n(a_j)\leq\lambda_n(a_i)+2$, a contradiction. Symmetry gives the same conclusion for $j<i$.

\smallskip \noindent Finally, let $\rho:B\to B_0$ be the retraction killing the other cyclic free factors. If $a_i^{-1}a_j=x^ny^n$ for $x,y\in B$, applying $\rho$ would express the same element as a product of two $n$th powers in $B_0$, which we have just excluded. Thus the sequence has the required property in $B$.
\end{proof}

	\begin{lemma}\label{dihedral_artin}
Let $H=A(m)$ be the Artin group of dihedral type with $m\in\{3,4,\ldots\}\cup\{\infty\}$, and let $N=Z(H)$. Then $H/N$ is a free product of cyclic groups which is not virtually abelian.
\end{lemma}

	\begin{proof}
If $m=\infty$, then $H$ is the non-abelian free group of rank two and $N=\{e\}$, so the conclusion is immediate. Henceforth assume $m<\infty$ and put $B=H/Z(H)$. We distinguish two cases.
\newline \underline{Case 1}. $m$ is even.
\newline In this case $H=A(m)$ is the Baumslag--Solitar group $BS(m/2,m/2)$ and, by \cite[Corollary~8.2]{levitt}, $B$ is a free product of cyclic groups. In particular,
\[H=\langle a,b\mid (ab)^{m/2}=(ba)^{m/2}\rangle
=\langle x,y\mid y^{-1}x^{m/2}y=x^{m/2}\rangle,\]
with $x=ab$ and $y=a$. The center of $H$ is the cyclic subgroup generated by the Garside element $x^{m/2}$ (see Theorem~\ref{censphAr} and \cite{Hum90}), which implies that $B\cong C_{m/2}*\mathbb Z$.
\newline \underline{Case 2}. $m$ is odd.
\newline In this case
\[H=\langle a,b\mid \underbrace{ab\cdots a}_{m\text{ letters}}=
\underbrace{ba\cdots b}_{m\text{ letters}}\rangle
=\langle x,y\mid x^2=y^m\rangle,\]
where $x=\underbrace{ab\cdots a}_{m\text{ letters}}$ and $y=ab$. Here $x$ is the standard Garside element, and by Theorem~\ref{censphAr} the center is generated by $x^2$. Therefore $B\cong C_2*C_m$.
In either case, $B$ is not virtually abelian.
\end{proof}

	\begin{corollary}\label{cor_super_tech}
Let $(A,S)$ be a non-abelian Artin system and let $\{a,b\}\subseteq S$ have label in $\{3,4,\ldots\}\cup\{\infty\}$. Put $H=\langle a,b\rangle_A$ and $N=Z(H)$, and let $F\trianglelefteq H/N$ be a non-abelian free subgroup of finite index $d$. If $H$ is $n$-pure in $A$ for some prime $n>d$, then $A$ is not superstable.
\end{corollary}

	\begin{proof}
By Lemma~\ref{dihedral_artin}, the quotient $P=H/N$ is a non-elementary virtually free group. Proposition~\ref{simon_prop}, since $n\geq2$, supplies a sequence of nonidentity elements of $P$ whose pairwise differences are not products of two $n$th powers. Choose lifts $a_\ell\in H\setminus N$ of these elements. Lemma~\ref{intermediate_lemma} now applies and yields non-superstability of $A$.
\end{proof}

	\begin{proof}[Proof of \ref{main_th2}] This is now immediate from  Corollary~\ref{cor_super_tech}.
	\end{proof}

\section{Elementary equivalence in spherical Artin groups}\label{sec_elem_eq_spher}

In this section we prove that two spherical Artin groups are elementarily equivalent if and only if they are isomorphic.

	\begin{lemma}\label{easy_lemma} Let $G$ and $H$ be groups. If $G$ and $H$ are elementarily equivalent, then so are $G/Z(G)$ and $H/Z(H)$.
\end{lemma}

	\begin{proof} This is folklore and it follows from the fact that $G/Z(G)$ is interpretable in $G$ and so it is a sort in $G^{\mrm{eq}}$, combined with the fact that if $G$ and $H$ are elementarily equivalent, then $G^{\mrm{eq}}$ and $H^{\mrm{eq}}$ are also elementarily equivalent.
	\end{proof}
	
	The following remark will play a crucial role in the proof of Theorem~\ref{spherical_irr_th}.
	
	\begin{remark}\label{useful_remark}
Suppose that $G=\langle s_1,\ldots,s_n\mid r_1,\ldots,r_m\rangle$ is finitely presented and has finitely many conjugacy classes of finite subgroups. Let $F_1,\ldots,F_t$ represent the conjugacy classes of nontrivial finite subgroups, and let $w_{i,1}(s),\ldots,w_{i,k_i}(s)$ represent the nonidentity elements of $F_i$, where $k_i=|F_i|-1$.
Consider the following formula $\theta(y)$, where $y$ is an $n$-tuple of variables:
\[
\begin{aligned}
\theta(y):={}&\bigwedge_{h=1}^{m}r_h(y)=e\ \wedge\ \bigwedge_{i=1}^{t}\bigwedge_{j=1}^{k_i}w_{i,j}(y)\neq e\\
&{}\wedge\ \forall g\,\bigwedge_{\substack{1\leq i,j\leq t\\i\neq j,\ k_i=k_j}}\ \bigvee_{q=1}^{k_i}\bigwedge_{l=1}^{k_j}gw_{i,q}(y)g^{-1}\neq w_{j,l}(y).
\end{aligned}
\]
Then for any group $G_1$ and $n$-tuple $(\gamma_1,\ldots,\gamma_n)\in G_1^n$, $G_1\models\theta(\gamma_1,\ldots,\gamma_n)$ if and only if the function $s_i\mapsto\gamma_i$ extends to a homomorphism $\phi:G\to G_1$ that is injective on every finite subgroup of $G$ and sends non-conjugate finite subgroups of $G$ to non-conjugate finite subgroups of $G_1$. It suffices to consider pairs with $k_i=k_j$, since injective images of finite subgroups of different orders cannot be conjugate.
\end{remark}

	\begin{theorem}\label{spherical_irr_th} If two irreducible spherical Artin groups are elementarily equivalent, then they are isomorphic.
\end{theorem}

\begin{proof} If either group has type $A_1$, then elementary equivalence forces the other group to be abelian as well. Since the only abelian irreducible spherical Artin group is $A(A_1)\cong\mathbb Z$, the conclusion follows. We may therefore assume that both groups are non-abelian.
Let $A$ and $B$ be as in the statement of the theorem. Clearly if $A$ is elementarily equivalent to $B$, then, by Lemma \ref{easy_lemma}, $A/Z(A)$ is elementarily equivalent to $B/Z(B)$.  Nearly all pairs $A,B$ can be distinguished by comparing the orders of the torsion elements modulo the centers. The strategy is to show that there are torsion elements in either $A':=A/Z(A)$ or $B':=B/Z(B)$ whose order is not admissible in the other group, using Table \ref{ord_tors_elem_table}. For the sake of clarity, we now give a few examples.
\begin{enumerate}[(1)]
\item[$\bullet$] Let $A=A(A_n)$ and $B=A(B_k)$ for some $n, k\geq 2$. If $k < n$, in $A'$ there is an element of order $n$; however, an element of order $n$ cannot exist in $B'$ as all the torsion elements in $B'$ have order $\leq k$. If $k = n$, then $A'$ admits an element of order $n+1$ and $B'$ does not.
If $k>n+1$, then $B'$ admits an element of order $k$ and $A'$ does not. Lastly, if $k=n+1$, we can find in $A'$ an element of order $n$, which is not possible in $B'$ as $n\nmid n+1$.
\item[$\bullet$] Let $A$ be as above and $B=A(D_k)$ for $k$ even and $>4$.  We recall that $A'$ has torsion elements of orders all the divisors of $n$ and $n+1$, while in $B'$ torsion elements have orders all the divisors of $k-1$ and $k/2$. As $k>4$, $\lvert (k-1)-k/2\rvert = k/2 -1 >1$. For $n<k-2$, both $n$ and $n+1$ are $<k-1$, therefore $B'$ admits an element of order $k-1$ and $A'$ does not. If $n=k-2$, then $n>k/2$ and $n$ is not a divisor of $k-1$, which implies that $A'$ has an element of order $n$ and $B'$ does not. Lastly, for $n>k-2$ we can take in $A'$ an element of order $n+1$.
\item[$\bullet$] Again, let $A$ be as above and $B=A(H_4)$. If either $n=2$ or $n=5$, the order that distinguishes $A'$ and $B'$ can be taken to be 15. If $n$ is neither 2 nor 5, then $\{n,n+1\}\not\subseteq\{2,3,5,6,10,15\}=:X$. Choose an element of $\{n,n+1\}$ that does not belong to $X$.
\item[$\bullet$] Let $A=A(B_n)$ and $B=A(F_4)$. The following torsion orders distinguish $A'$ from $B'$: for $n<6$, take 6; for $n=6$, take 4; for $n>6$, take $n$.
\end{enumerate}
In a similar fashion, by comparing the orders of the torsion elements in $A/Z(A)$ and $B/Z(B)$ we are able to distinguish all pairs of irreducible spherical Artin groups except the following cases:
\begin{enumerate}[(i)]
\item\label{eleq1} $A_2$ and $D_4$;
\item\label{eleq2} $B_n$ and $I_2(2n)$, for $n \geq 3$;
\item\label{eleq3} $D_6$ and $H_3$.
\end{enumerate}
Let $A$ be either $A(A_2)$ or $A(I_2(2n))$ for $n \geq 3$ and $B$ be either $A(D_4)$ or $A(B_n)$ respectively, as in cases (\ref{eleq1}) and (\ref{eleq2}). Then $A/Z(A)$ is hyperbolic, as by Lemma~\ref{dihedral_artin} it is a free product of cyclic groups that is not virtually abelian (analogous calculations hold for $A(A_2)$).
On the other hand, $B/Z(B)$ is not hyperbolic, as we are going to show that it embeds $\mathbb{Z}^2$. By \cite[Theorem 1.2]{andre_hyper},
if two finitely generated groups are elementarily equivalent, either they are both hyperbolic or none of them is. Therefore, $A/Z(A)$ and $B/Z(B)$ are not elementarily equivalent in cases (\ref{eleq1}) and (\ref{eleq2}).
We only need to show that $B/Z(B)$ embeds $\mathbb{Z}^2$, and we will only do it for $B=A(D_4)$ as the same argument holds for $A(B_n)$. The Coxeter graph of type $D_4$ is shown in Figure~\ref{d4graph}.
\begin{figure}
		\begin{tikzpicture}
		\draw[thick] (0,0) -- (1,1) -- (2.5,1)
								(0,2) -- (1,1);
								
		\node[circle, draw, fill, text width=1mm, inner sep=0.5] at (0,0) {};
		\node[circle, draw, fill, text width=1mm, inner sep=0.5] at (0,2) {};
		\node[circle, draw, fill, text width=1mm, inner sep=0.5] at (1,1) {};
		\node[circle, draw, fill, text width=1mm, inner sep=0.5] at (2.5,1) {};
		
		\node[below left, font=\small] at (0,0) {$s_2$};
		\node[below left, font=\small] at (0,2) {$s_1$};
		\node[below right, font=\small] at (1,1) {$s_3$};
		\node[below right, font=\small] at (2.5,1) {$s_4$};
	\end{tikzpicture}
	\caption{The Coxeter graph of type $D_4$}
\label{d4graph}
\end{figure}
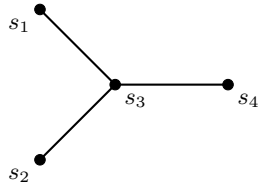
Let $a,b$ be two generators of $A(D_4)$ corresponding to two vertices of degree 1 in $D_4$; for instance, let $a=s_1$ and $b=s_2$. By \cite{vdL83}, as we recalled in Section~\ref{prelim}, the standard parabolic subgroup $H$ generated by $a$ and $b$ is isomorphic to $\mathbb{Z}^2$. The image of $H$ in the quotient $B/Z(B)$ is the subgroup $HZ(B)/Z(B)\cong H/H\cap Z(B)$ (the last isomorphism being an application of the Second Isomorphism Theorem for groups).
Suppose that $H\cap Z(B)\ne \{e\}$. Then $\Delta^k\in H$ for some $k\ne 0$, where $\Delta$ is the standard Garside element of $B$, as $Z(B)$ is generated by a power of  $\Delta$. This implies, by Proposition~\ref{rootparsgb}, that $\Delta\in H$, which is impossible (see Proposition~\ref{delparsgb}). Therefore, the image of $H$ in $B/Z(B)$ is isomorphic to $H$. This completes the proof in cases (\ref{eleq1}) and (\ref{eleq2}).

\smallskip \noindent We are then left to deal with case (\ref{eleq3}). Let $A:=A(D_6)/Z(A(D_6))$ and $B:=A(H_3)/Z(A(H_3))$. The Coxeter graphs of type $D_6$ and $H_3$ are shown in Figure~\ref{d6h3graphs}.

\begin{figure}
	\begin{tikzpicture}
		\node () at (-1,1) {$D_6:$};
		\draw[thick] (0,0) -- (1,1) -- (4,1)
		(0,2) -- (1,1);
		\node[circle, draw, fill, text width=1mm, inner sep=0.5] at (0,0) {};
		\node[circle, draw, fill, text width=1mm, inner sep=0.5] at (0,2) {};
		\node[circle, draw, fill, text width=1mm, inner sep=0.5] at (1,1) {};
		\node[circle, draw, fill, text width=1mm, inner sep=0.5] at (2,1) {};
		\node[circle, draw, fill, text width=1mm, inner sep=0.5] at (3,1) {};
		\node[circle, draw, fill, text width=1mm, inner sep=0.5] at (4,1) {};
		
		\node[below, font=\small] at (0,2) {$t_1$};
		\node[below, font=\small] at (0,0) {$t_2$};
		\node[below, font=\small] at (1,1) {\,$t_3$};
		\node[below, font=\small] at (2,1) {$t_4$};
		\node[below, font=\small] at (3,1) {$t_5$};
		\node[below, font=\small] at (4,1) {$t_6$};

		\node () at (6,1) {$H_3:$};
		\draw[thick] (7,1) -- (9,1);
		
		\node[circle, draw, fill, text width=1mm, inner sep=0.5] at (7,1) {};
		\node[circle, draw, fill, text width=1mm, inner sep=0.5] at (8,1) {};
		\node[circle, draw, fill, text width=1mm, inner sep=0.5] at (9,1) {};
		
		\node[below, font=\small] at (7,1) {$s_1$};
		\node[below, font=\small] at (8,1) {$s_2$};
		\node[below, font=\small] at (9,1) {$s_3$};
		\node[above, font=\small] at (7.5,1) {5};
	\end{tikzpicture}
	\caption{The Coxeter graphs of type $D_6$ and $H_3$}
	\label{d6h3graphs}
\end{figure}
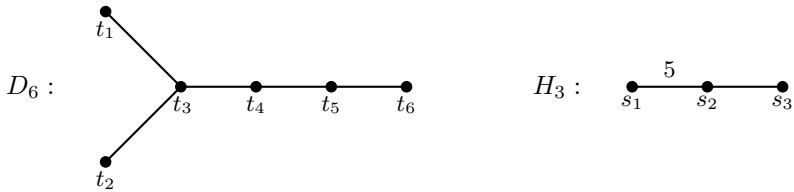

\smallskip \noindent By \cite{Sor21}, every nontrivial torsion element in $A$ or $B$ has order either $3$ or $5$, and by \cite{Par04} the maximal order of a finite subgroup of $A$ or $B$ is $5$. As a consequence, the only admissible nontrivial finite subgroups of $A$ or $B$ are cyclic groups of order $3$ or $5$.
Moreover, by \cite{Sor21}, every torsion element of $A$ is conjugate to a power of one of the so-called basic torsion elements listed in column 3 of \cite[Table 1]{Sor21}, and the same holds for $B$. We can therefore deduce that both $A$ and $B$ have a finite number of conjugacy classes of finite subgroups.

\smallskip \noindent Suppose now that $A\equiv B$. Then, using \ref{useful_remark}, we can find a group homomorphism $\phi\colon A\to B$ that satisfies the properties from \ref{useful_remark}. Analogously, there is a group homomorphism $\phi'\colon B\to A$ that satisfies the properties from \ref{useful_remark}. By composing them, we find an endomorphism $\psi:=\phi'\circ\phi\colon A\to A$ which is injective on finite subgroups and sends pairs of non-conjugate finite subgroups to pairs of non-conjugate finite subgroups. As the number of conjugacy classes of finite subgroups of $A$ is finite, there exists an $N>0$ such that $\chi:=\psi^N$ sends every finite subgroup of $A$ to a conjugate one. Moreover, $\chi$ is injective on the finite subgroups of $A$.

\smallskip \noindent By \cite[Theorem~2.8]{castel}, either $\chi$ is cyclic (i.e., its  image is a cyclic subgroup of $A$), or it is an automorphism of $A$. Suppose it is cyclic.
Suppose that its image $D$ is an infinite cyclic group; then we reach a contradiction as $D$ is torsion-free and therefore this would contradict the injectivity of $\chi$ on finite subgroups of $A$.
If $D$ were finite cyclic, injectivity on finite subgroups would force $D$ to contain elements of both orders $3$ and $5$, hence an element of order $15$. This is impossible, because the only nontrivial torsion orders in $A$ are $3$ and $5$.
Therefore $\chi$ cannot be a cyclic endomorphism, which implies that $\chi$ and $\psi$ are automorphisms. It follows that $\phi$ is injective and $\phi'$ is surjective.

\smallskip \noindent By Theorem \ref{censphAr} and Remark \ref{fundam_elem}, $A$ is the quotient of $A(D_6)$ by the normal closure of the Garside element $\Delta_A:=(t_1\dots t_6)^5$, hence a presentation for $A$ is given by the presentation of $A(D_6)$ with the extra relation $\Delta_A=e$. The presentation of the abelianization is therefore obtained by imposing commutativity among the generators, which implies that in the abelianization $t_1=\dots=t_6=:t$ and $t^{30}=e$, i.e., the abelianization of $A$ is isomorphic to $\mathbb{Z}/30\mathbb{Z}$ . Similarly, with $\Delta_B:=(s_1s_2s_3)^5$, the abelianization of $B$ is isomorphic to $\mathbb{Z}/15\mathbb{Z}$. Since $\psi=\phi'\circ\phi$ is an automorphism, the map $r=\psi^{-1}\circ\phi':B\to A$ satisfies $r\circ\phi=\mathrm{id}_A$. Passing to abelianizations gives $r_{\mathrm{ab}}\circ\phi_{\mathrm{ab}}=\mathrm{id}_{A_{\mathrm{ab}}}$. Thus $A_{\mathrm{ab}}\cong\mathbb Z/30\mathbb Z$ embeds in $B_{\mathrm{ab}}\cong\mathbb Z/15\mathbb Z$, a contradiction.  Therefore $A\not\equiv B$, which concludes the proof.
\end{proof}

\begin{fact}[{\cite{alg_geom_over_groups3}}]\label{fact_unique_dec} Let $G = G_1 \times \cdots \times G_n$ be a finite direct product of nontrivial domains $G_1$, \dots, $G_n$. Then $G$ has
a unique (up to a permutation of factors) finite direct product decomposition into nontrivial indecomposable
groups.
\end{fact}

\begin{fact}[{\cite{alg_geom_over_groups3}}]\label{domain_fact}
Let $G$ be a finite direct product of nontrivial domains $G_1, \ldots, G_n$. If $G \equiv H$, then $H$ is also a finite direct product of nontrivial domains.
In addition, if
$G = G_1 \times \ldots \times G_k$ and $H = H_1 \times \ldots \times H_m$
are their decompositions into nontrivial domains, we have $k = m$ and $G_i \equiv H_i$ (after a suitable reordering of the factors).
\end{fact}

We recall here that, as we saw in Fact \ref{center_dec_fact}, the center $Z(A)$ of a spherical Artin group $A$ is given by the direct product $Z(A_1)\times\cdots\times Z(A_n)$ of the centers of its irreducible components $A_1,\dots,A_n$.

	\begin{fact}[{\cite{CalvezWiest17}}]\label{acy_spher}
Let $A$ be a non-abelian irreducible spherical Artin group. Then $A/Z(A)$ is acylindrically hyperbolic.
\end{fact}

	\begin{proof}[Proof of Theorem~\ref{main_th3}]
Suppose $A\equiv B$. Write \[A\cong\mathbb Z^r\times A_1\times\cdots\times A_s,\qquad B\cong\mathbb Z^t\times B_1\times\cdots\times B_u,\] where all the factors $A_i$ and $B_j$ are non-abelian irreducible spherical Artin groups. By Lemma~\ref{easy_lemma}, $A/Z(A)\equiv B/Z(B)$. Each nontrivial factor $A_i/Z(A_i)$ or $B_j/Z(B_j)$ is a domain by Corollary~\ref{spherical_domain}. If $s=0$, then $A/Z(A)$ is trivial, so $B/Z(B)$ is trivial and $u=0$. Otherwise Fact~\ref{domain_fact} and uniqueness of the decomposition into nontrivial domains give $s=u$ and, after reordering, $A_i/Z(A_i)\equiv B_i/Z(B_i)$ for every $i$. The proof of Theorem~\ref{spherical_irr_th} distinguishes these central quotients for nonisomorphic non-abelian irreducible spherical Artin groups. Hence $A_i\cong B_i$ for every $i$.

\smallskip \noindent By Theorem~\ref{censphAr}, $Z(A)\cong\mathbb Z^{r+s}$ and $Z(B)\cong\mathbb Z^{t+s}$. The center is defined without parameters by $z(x):=\forall y\,(xy=yx)$, and its subgroup of squares is defined by $q(x):=\exists y\,(z(y)\wedge x=y^2)$. Thus $Z(A)/Z(A)^2$ and $Z(B)/Z(B)^2$ are uniformly interpretable. The assertion that this quotient has exactly $M$ elements is first-order: it asserts that there are $M$ central representatives with pairwise distinct cosets and that every central element belongs to one of these cosets. Elementary equivalence therefore gives $2^{r+s}=2^{t+s}$, whence $r=t$ and $A\cong B$. The converse follows immediately from isomorphism.
\end{proof}

\section{Elementary equivalence in affine Artin groups of type \texorpdfstring{$\tilde{A}_n$}{A-tilde-n}}

This section is devoted to distinguishing affine Artin groups of type $\tilde{A}_n$ with $n \geq 4$ from the other simply laced affine Artin groups.
To start, we recall the following definition, introduced by O.~Houcine in \cite{houcine2}.

\begin{definition}A group $G$ is said to be strongly co-Hopfian if there exists a finite subset $F \subset G \setminus \lbrace 1\rbrace$ such that, for any endomorphism $\phi$ of $G$, if $\ker(\phi) \cap  F = \varnothing$ then $\phi$ is an automorphism.
\end{definition}

We will need the following lemma.

\begin{lemma}
Affine Artin groups of type $\tilde{A}_n$ for $n\geq 4$ are strongly co-Hopfian.
\end{lemma}

\begin{proof}
Let $G=A(\tilde A_n)$ with $n\geq4$. By \cite[Theorem~2.1]{tilde_An}, for any endomorphism $\varphi$ of $G$, either $\varphi$ is an automorphism of $G$ or the image of $\varphi$ is abelian or $\varphi=\psi\circ \alpha_p$ or $\varphi=\psi\circ \beta_p$ for some automorphism $\psi$ of $G$ and some integer $p$, where $\alpha_p$ and $\beta_p$ are two endomorphisms of $G$ defined in \cite{tilde_An}. These endomorphisms $\alpha_p$ and $\beta_p$ enjoy the following property: there exist $a,b\in G\setminus \lbrace 1\rbrace$ that do not depend on $p$ such that $\alpha_p(a)=1$ and $\beta_p(b)=1$ (for instance, one can take $a=t_0^{-1}\nu_0$ and $b=t_0^{-1}\nu_1$ using the same notation as in \cite{tilde_An}). Now, let $x,y\in G$ be such that $g=[x,y]$ is nontrivial, and define $F=\lbrace g,a,b\rbrace$. Therefore any endomorphism $\varphi$ of $G$ such that $\ker(\varphi)\cap F=\varnothing$ is an automorphism.
\end{proof}

\begin{corollary}\label{affArt_retr}
Let $G$ be the Artin group of type $\tilde{A}_n$ for $n\geq 4$. Let $G'$ be a finitely presented group. If $\mathrm{Th}_{\exists}(G)=\mathrm{Th}_{\exists}(G')$ then there exists an embedding $i : G\hookrightarrow G'$ such that $i(G)$ is existentially closed in $G'$. Moreover, $G'$ retracts onto $i(G)$. 
\end{corollary}

\begin{proof}
Both $G$ and $G'$ are finitely presented, and $G$ is strongly co-Hopfian by the preceding lemma. Hence \cite[Lemma~4.2]{andre_exclos} gives an embedding $i:G\hookrightarrow G'$ whose image is existentially closed. The proof of that lemma constructs retractions from $G'$ onto $i(G)$, giving the final assertion.
\end{proof}

For later use, we recall that $\operatorname{cd}A(\tilde A_j)=j$ for every $j\geq2$, where $\operatorname{cd}$ denotes integral cohomological dimension. Indeed, the Salvetti complex is a $K(\pi,1)$ by \cite{P-Salvetti21} and has dimension $j$, giving the upper bound. A standard parabolic subgroup of type $A_j$ contains $\mathbb Z^j$: in its braid realization, the Dehn twists about $j$ nested curves enclosing respectively $2,\ldots,j+1$ punctures generate a free abelian group of rank $j$. Since cohomological dimension cannot increase on passing to a subgroup, this gives the lower bound.

\begin{corollary}\label{tildeA_exequiv}
Two Artin groups of type $\tilde{A}_n$ and $\tilde{A}_m$ with $n,m\geq 4$ are existentially equivalent if and only if $n=m$.
\end{corollary}

\begin{proof}
By Corollary~\ref{affArt_retr}, existential equivalence gives embeddings in both directions. Since $\operatorname{cd}A(\tilde A_j)=j$ and cohomological dimension cannot increase on passing to a subgroup, we obtain $n\leq m$ and $m\leq n$. Hence $n=m$. The converse is immediate.
\end{proof}

Corollary \ref{affArt_retr} provides the injection and retraction necessary to apply the techniques previously employed in the spherical case and for affine Artin groups of type $\tilde{A}_n$.
However, to show that such an injection or retraction cannot exist when dealing with other affine Artin groups, a different approach is required. Here, we use homology to provide an obstruction.

Recall that every Artin system $(A, S)$ has an associated \textit{orbit configuration space} which is conjectured to be a $K(A, 1)$ space \cite{ParConj}.
This conjecture, called the $K(\pi, 1)$ conjecture, is known to hold for affine Artin groups \cite{P-Salvetti21}.
The orbit configuration space of $(A, S)$ has the homotopy type of a CW complex known as the Salvetti complex \cite{Salvetti87,Salvetti94}.
The algebraic chain complex $C_*$ that computes the cellular homology of the Salvetti complex can therefore be used to compute the homology of an affine Artin group $A$.
Recall from \cite{Salvetti94} that the standard basis of $C_*$ is indexed by the following family of subsets of $S$:
\[ K = \{ \sigma \subseteq S \mid \text{the parabolic subgroup $A_\sigma$ generated by $\sigma$ is spherical} \};  \]
note that $K$ can be thought of as a simplicial complex on the vertex set $S$.
Specifically, there is a basis element $e_\sigma \in C_{|\sigma|}$ for every $\sigma \in K$, and the boundary $\partial$ is computed as follows:
\begin{equation}
   \partial e_\sigma = \sum_{\substack{\tau\subset\sigma\\|\tau|=|\sigma|-1}}[\sigma:\tau]\left.\frac{W_\sigma(q)}{W_\tau(q)}\right|_{q=-1}\,e_\tau
   \label{eq:boundary}
\end{equation}
where the sum is over all codimension-one faces $\tau$ of $\sigma$.
By $[\sigma : \tau] \in \{\pm 1\}$ we denote the incidence number between the simplices $\sigma$ and $\tau$, and by $W_\sigma(q) \in \mathbb{N}[q]$ we denote the Poincaré polynomial of the (finite) Coxeter system associated with the Artin system $(A_\sigma, \sigma)$ (see e.g.\ \cite{Hum90} for how to compute the Poincaré polynomials).
The notation with the vertical bar means that we evaluate at $q=-1$ the quotient $\frac{W_\sigma(q)}{W_\tau(q)}$, which is always a polynomial with nonnegative integer coefficients.

To find the required obstruction, we are going to compute $H_2(A; \mathbb{Z})$ where $A$ is an affine Artin group of type $\tilde{A}_n$, $\tilde{D}_n$, or $\tilde{E}_n$ (see Figure~\ref{irredaffineCoxdiagr}).
We perform this computation for any simply laced irreducible Artin group (not necessarily affine), under the assumption that the $K(\pi, 1)$ conjecture holds.

The following statement is the simply laced irreducible case of the Clancy--Ellis formula for the second homology of the orbit configuration space, recalled in \cite[Theorem~2.6]{AkitaLiu18}, together with the $K(\pi,1)$ assumption. We include a direct computation to make the invariant $c$ explicit.

\begin{theorem}
    Let $\Gamma$ be a simply laced irreducible Coxeter graph, and suppose that the $K(\pi, 1)$ conjecture holds for the Artin group $A(\Gamma)$. Then $H_2(A(\Gamma); \mathbb{Z}) \cong \mathbb{Z}^b \oplus (\mathbb{Z} / 2\mathbb{Z})^c$, where $b \in \mathbb{N}$ is the first Betti number of the graph $\Gamma$ and $c \in \mathbb{N}$ is another natural number which can be explicitly computed from $\Gamma$.
    In particular, if $\Gamma$ is a tree, then $H_2(A(\Gamma); \mathbb{Z}) \cong (\mathbb{Z} / 2\mathbb{Z})^c$.
    \label{thm:second-homology}
\end{theorem}

\begin{proof}
    To compute $H_2(A; \mathbb{Z})$, we need the following portion of the aforementioned chain complex $C_*$:
    \[ C_3 \xrightarrow{\;\partial_3\;} C_2 \xrightarrow{\;\partial_2\;} C_1. \]
    Table \ref{table:poincare-polynomials} lists all values for the Poincaré polynomials appearing in the computation of $\partial_3$ and $\partial_2$ according to \eqref{eq:boundary}.
    
    \begin{table}
        \begin{tabular}{cc}
            \toprule
            Coxeter subgraph & $W_\sigma(q)$  \\
            \toprule
            $A_1$ & $1 + q$ \\
            \midrule
            $A_1 \sqcup A_1$ & $(1 + q)^2$ \\
            $A_2$ & $(1 + q)(1 + q + q^2)$ \\
            \midrule
            $A_1 \sqcup A_1 \sqcup A_1$ & $(1 + q)^3$ \\
            $A_2 \sqcup A_1$ & $(1 + q)^2(1 + q + q^2)$ \\
            $A_3$ & $(1 + q)^2(1 + q + q^2)(1 + q^2)$ \\
            \bottomrule
        \end{tabular}
        \medskip
        \caption{Poincaré polynomials $W_\sigma(q)$ as a function of the type of the Coxeter subgraph induced by $\sigma$}
        \label{table:poincare-polynomials}
    \end{table}
    
    \smallskip \noindent Let $s_1, \dots, s_n$ be the generators of $A(\Gamma)$, which correspond to the vertices of $\Gamma$.
    To simplify the notation, denote by $e_{i}$, $e_{ij}$, and $e_{ijk}$ the basis elements $e_{\{s_i\}}$, $e_{\{s_i, s_j\}}$, and $e_{\{s_i, s_j, s_k\}}$, respectively, for $1 \leq i < j < k \leq n$. The notation $e_{ijk}$ is used only when $\{s_i,s_j,s_k\}\in K$, that is, when the corresponding rank-three parabolic subgroup is spherical. 
    It is convenient to define $e_{ij}$ and $e_{ijk}$ also when the indices are not in increasing order, by skew symmetry: $e_{ij} = -e_{ji}$; $e_{ijk} = e_{jki} = e_{kij} = -e_{jik} = -e_{ikj} = -e_{kji}$.

    \smallskip \noindent Using \eqref{eq:boundary}, we obtain:
    \[ \partial e_{ij} = \begin{cases}
        0 & \text{if $m_{ij} = 2$} \\
        e_i - e_j & \text{if $m_{ij} = 3$}.
    \end{cases} \]
    Therefore, we have that $\ker \partial_2 \cong H_1(\Gamma; \mathbb{Z}) \oplus N$, where $N$ is the free $\mathbb{Z}$-module generated by all $e_{ij}$ with $m_{ij} = 2$; in other words, the generators of $N$ are in bijection with the non-edges of the Coxeter graph.

    \smallskip \noindent We now examine the image of the generators $e_{ijk}$ of $C_3$.
    If $m_{ij} = m_{ik} = m_{jk} = 2$ (i.e., there is no edge between the three vertices $s_i, s_j, s_k$), then $\partial e_{ijk} = 0$.
    If $m_{ij} = 3$ and $m_{ik} = m_{jk} = 2$, then $\partial e_{ijk} = e_{jk} - e_{ik}$.
    Therefore, two generators $e_{ik}$ and $e_{jk}$ of $N$ are identified in $H_2(A; \mathbb{Z})$ whenever there is an edge connecting $s_i$ and $s_j$ in $\Gamma$.
    Up to permutation of the indices, there is only one remaining case, namely $m_{ij} = 2$ and $m_{ik} = m_{jk} = 3$; in this case, we have $\partial e_{ijk} = 2e_{ij}$.
    Therefore, a generator $e_{ij}$ of $N$ satisfies $2 [e_{ij}] = 0$ in $H_2(A; \mathbb{Z})$ whenever the vertices $s_i$ and $s_j$ have distance $2$ in the Coxeter graph.
    
    \smallskip \noindent If $s_i$ and $s_j$ have distance $d \geq 3$ in the Coxeter graph, and $s_{i_0}, s_{i_1}, \dots,s_{i_d}$ is a shortest path between them with $i_0 = i$ and $i_d = j$, then $[e_{ij}] = [e_{i_1j}] = \dots = [e_{i_{d-2}j}]$; here, the vertices $s_{i_{d-2}}$ and $s_j$ have distance $2$, and therefore $2[e_{ij}] = 2[e_{i_{d-2}j}] = 0$.
    We conclude that every generator $e_{ij}$ of $N$ in fact satisfies $2 [e_{ij}] = 0$ in $H_2(A;\mathbb{Z})$.

    \smallskip \noindent Note that $\text{im}\, \partial_3 \subseteq N$, and therefore
    $H_2(A; \mathbb{Z}) \cong H_1(\Gamma; \mathbb{Z}) \oplus ( N / \text{im}\, \partial_3 )$.
    The first summand is a free $\mathbb{Z}$-module of rank equal to the first Betti number of~$\Gamma$.
    The second summand is of the form $(\mathbb{Z}/2\mathbb{Z})^c$, where $c$ is the number of equivalence classes of non-edges with respect to the equivalence relation generated by the following identifications: $\{s_i, s_k\} \sim \{ s_j, s_k \}$ whenever $\{s_i, s_j\}$ is an edge.
\end{proof}

We record the resulting affine values for use below. The $\tilde D_n$ and $\tilde E_n$ values already appear in \cite[Corollary~3.21 and Example~3.24]{AkitaLiu18}, independently of the affine $K(\pi,1)$ conjecture.

\begin{proposition}
    Let $A$ be an affine Artin group of type $\tilde A_n$, $\tilde D_n$, or $\tilde E_n$.
    Then
    \[
        H_2(A; \mathbb{Z}) \cong
        \begin{cases}
            0 & \text{in the case $\tilde A_1$} \\
            \mathbb{Z} & \text{in the case $\tilde A_2$} \\
            \mathbb{Z} \oplus (\mathbb{Z} / 2 \mathbb{Z})^2 & \text{in the case $\tilde A_3$} \\
            \mathbb{Z} \oplus \mathbb{Z} / 2 \mathbb{Z} & \text{in the case $\tilde A_n$ with $n \geq 4$} \\
            (\mathbb{Z} / 2 \mathbb{Z})^6 & \text{in the case $\tilde D_4$} \\
            (\mathbb{Z} / 2 \mathbb{Z})^3 & \text{in the case $\tilde D_n$ with $n\geq 5$} \\
            \mathbb{Z} / 2 \mathbb{Z} & \text{in the case $\tilde E_n$ with $n=6,7,8$}. \\
        \end{cases}
    \]
    \label{prop:affine-homology}
\end{proposition}

\begin{figure}
	\centering
	\begin{tikzpicture}
		\node () at (0,5) {$\tilde{A}_1:$};
		\draw[thick] (2,5) -- (3,5);
		\node[circle, draw, fill, text width=1mm, inner sep=0.5] at (2,5) {};
		\node[circle, draw, fill, text width=1mm, inner sep=0.5] at (3,5) {};
		
		\node[below, font=\small] at (2,5) {$s_1$};
		\node[below, font=\small] at (3,5) {$s_2$};
        \node[above, font=\small] at (2.5,5) {$\infty$};

		\node () at (0,4) {$\tilde{A}_n$\;$(n\geq 2):$};
		\draw[thick] (2,4) -- (4,4)
		(5,4) -- (6,4)
        (2,4) -- (4,4.5)
        (6,4) -- (4,4.5);
		\draw[dashed, thick] (4,4) -- (5,4);
		\node[circle, draw, fill, text width=1mm, inner sep=0.5] at (2,4) {};
		\node[circle, draw, fill, text width=1mm, inner sep=0.5] at (3,4) {};
		\node[circle, draw, fill, text width=1mm, inner sep=0.5] at (4,4) {};
        \node[circle, draw, fill, text width=1mm, inner sep=0.5] at (4,4.5) {};
		\node[circle, draw, fill, text width=1mm, inner sep=0.5] at (5,4) {};
		\node[circle, draw, fill, text width=1mm, inner sep=0.5] at (6,4) {};
		
		\node[below, font=\small] at (2,4) {$s_1$};
		\node[below, font=\small] at (3,4) {$s_2$};
		\node[below, font=\small] at (4,4) {$s_3$};
		\node[below, font=\small] at (5,4) {$s_{n-1}$};
		\node[below, font=\small] at (6,4) {$s_n$};
        \node[above, font=\small] at (4,4.5) {$s_{n+1}$};

		\node () at (0,3) {$\tilde{D}_n$\;$(n\geq 4):$};
		\draw[thick] (2,3.3) -- (3,3) -- (4,3)
		(2,2.7) -- (3,3)
		(5,3) -- (6,3.3)
        (5,3) -- (6,2.7);
		\draw[dashed, thick] (4,3) -- (5,3);
		\node[circle, draw, fill, text width=1mm, inner sep=0.5] at (2,3.3) {};
		\node[circle, draw, fill, text width=1mm, inner sep=0.5] at (2,2.7) {};
		\node[circle, draw, fill, text width=1mm, inner sep=0.5] at (3,3) {};
		\node[circle, draw, fill, text width=1mm, inner sep=0.5] at (4,3) {};
		\node[circle, draw, fill, text width=1mm, inner sep=0.5] at (5,3) {};
		\node[circle, draw, fill, text width=1mm, inner sep=0.5] at (6,3.3) {};
        \node[circle, draw, fill, text width=1mm, inner sep=0.5] at (6,2.7) {};
		
		\node[below, font=\small] at (2,3.3) {$s_1$};
		\node[below, font=\small] at (2,2.7) {$s_2$};
		\node[below, font=\small] at (3,3) {$s_3$};
		\node[below, font=\small] at (4,3) {$s_4$};
		\node[below, font=\small] at (5,3) {$s_{n-1}$};
		\node[below, font=\small] at (6,3.3) {$s_n$};
        \node[below, font=\small] at (6,2.7) {$s_{n+1}$};

		\node () at (0,2) {$\tilde{E}_6:$};
		\draw[thick] (2,2) -- (6,2)
		(4,2) -- (4,1);
		
		\node[circle, draw, fill, text width=1mm, inner sep=0.5] at (2,2) {};
		\node[circle, draw, fill, text width=1mm, inner sep=0.5] at (3,2) {};
		\node[circle, draw, fill, text width=1mm, inner sep=0.5] at (4,2) {};
		\node[circle, draw, fill, text width=1mm, inner sep=0.5] at (5,2) {};
		\node[circle, draw, fill, text width=1mm, inner sep=0.5] at (6,2) {};
        \node[circle, draw, fill, text width=1mm, inner sep=0.5] at (4,1) {};
		\node[circle, draw, fill, text width=1mm, inner sep=0.5] at (4,1.5) {};
		
		\node[below, font=\small] at (2,2) {$s_1$};
		\node[right, font=\small] at (4,1.5) {$s_4$};
        \node[right, font=\small] at (4,1) {$s_5$};
		\node[below, font=\small] at (3,2) {$s_2$};
		\node[below left, font=\small] at (4,2) {$s_3$};
		\node[below, font=\small] at (5,2) {$s_6$};
		\node[below, font=\small] at (6,2) {$s_7$};

		\node () at (0,0) {$\tilde{E}_7:$};
		\draw[thick] (2,0) -- (8,0)
		(5,0) -- (5,-0.5);
		
		\node[circle, draw, fill, text width=1mm, inner sep=0.5] at (2,0) {};
		\node[circle, draw, fill, text width=1mm, inner sep=0.5] at (3,0) {};
		\node[circle, draw, fill, text width=1mm, inner sep=0.5] at (4,0) {};
		\node[circle, draw, fill, text width=1mm, inner sep=0.5] at (5,0) {};
		\node[circle, draw, fill, text width=1mm, inner sep=0.5] at (6,0) {};
		\node[circle, draw, fill, text width=1mm, inner sep=0.5] at (7,0) {};
        \node[circle, draw, fill, text width=1mm, inner sep=0.5] at (8,0) {};
		\node[circle, draw, fill, text width=1mm, inner sep=0.5] at (5,-0.5) {};
		
		\node[below, font=\small] at (2,0) {$s_1$};
		\node[right, font=\small] at (5,-0.5) {$s_5$};
		\node[below, font=\small] at (3,0) {$s_2$};
		\node[below, font=\small] at (4,0) {$s_3$};
		\node[below left, font=\small] at (5,0) {$s_4$};
		\node[below, font=\small] at (6,0) {$s_6$};
		\node[below, font=\small] at (7,0) {$s_7$};
        \node[below, font=\small] at (8,0) {$s_8$};

		\node () at (0,-1) {$\tilde{E}_8:$};
		\draw[thick] (2,-1) -- (9,-1)
		(4,-1) -- (4,-1.5);
		
		\node[circle, draw, fill, text width=1mm, inner sep=0.5] at (2,-1) {};
		\node[circle, draw, fill, text width=1mm, inner sep=0.5] at (3,-1) {};
		\node[circle, draw, fill, text width=1mm, inner sep=0.5] at (4,-1) {};
		\node[circle, draw, fill, text width=1mm, inner sep=0.5] at (5,-1) {};
		\node[circle, draw, fill, text width=1mm, inner sep=0.5] at (6,-1) {};
		\node[circle, draw, fill, text width=1mm, inner sep=0.5] at (7,-1) {};
		\node[circle, draw, fill, text width=1mm, inner sep=0.5] at (8,-1) {};
        \node[circle, draw, fill, text width=1mm, inner sep=0.5] at (9,-1) {};
		\node[circle, draw, fill, text width=1mm, inner sep=0.5] at (4,-1.5) {};
		
		\node[below, font=\small] at (2,-1) {$s_1$};
		\node[right, font=\small] at (4,-1.5) {$s_4$};
		\node[below, font=\small] at (3,-1) {$s_2$};
		\node[below left, font=\small] at (4,-1) {$s_3$};
		\node[below, font=\small] at (5,-1) {$s_5$};
		\node[below, font=\small] at (6,-1) {$s_6$};
		\node[below, font=\small] at (7,-1) {$s_7$};
		\node[below, font=\small] at (8,-1) {$s_8$};
        \node[below, font=\small] at (9,-1) {$s_9$};
	\end{tikzpicture}
	\caption{The Coxeter graphs of the affine Artin groups of type $\tilde{A}_n$, $\tilde{D}_n$, and $\tilde{E}_n$.}
	\label{irredaffineCoxdiagr}
\end{figure}
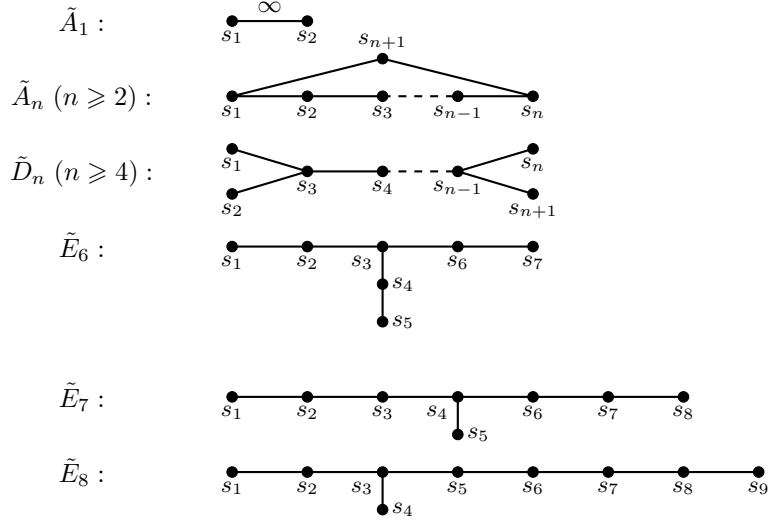

\begin{proof}
    Label the generators as in Figure \ref{irredaffineCoxdiagr}. The case $\tilde{A}_1$ is not simply laced, but $C_2 = 0$ and so its second homology group vanishes.
    In the other cases, we can use Theorem~\ref{thm:second-homology}.
    For $n \geq 2$, the $\tilde{A}_n$ graph is a cycle, so $H_1(\Gamma; \mathbb{Z}) \cong \mathbb{Z}$. This accounts for the $\mathbb{Z}$ summand in the second homology, which $\tilde{D}_n$ and $\tilde{E}_n$ do not have.

    \smallskip \noindent To explicitly compute the rank of the $2$-torsion part, we need to determine the number of equivalence classes of non-edges $\{s_i, s_j\}$, as described in the proof of Theorem~\ref{thm:second-homology}:
    \begin{itemize}
        \item in the case $\tilde{A}_2$, there are zero non-edges;
        \item in the case $\tilde{A}_3$, the two non-edges $\{s_1, s_3\}$ and $\{s_2, s_4\}$ are not equivalent;
        \item in the case $\tilde{A}_n$ for $n \geq 4$, all non-edges are equivalent to each other;
        \item for $\tilde{D}_4$, the six non-edges $\{s_1, s_2\}$, $\{s_1, s_4\}$, $\{s_1, s_5\}$, $\{s_2, s_4\}$, $\{s_2, s_5\}$, and $\{s_4, s_5\}$ are not equivalent;
        \item in the case $\tilde{D}_n$ for $n \geq 5$, all non-edges are equivalent to each other, except $\{s_1, s_2\}$ and $\{s_n, s_{n+1}\}$, for a total of three equivalence classes;
        \item in the case $\tilde{E}_n$, all non-edges are equivalent to each other. \qedhere
    \end{itemize}
\end{proof}

\begin{corollary}\label{tildeADE}
    Let $A$ be an Artin group of type $\tilde D_m$ or $\tilde E_m$.
    Then $A$ does not retract onto a copy of the Artin group $A(\tilde A_n)$ for any $n$.
\end{corollary}

\begin{proof}
    Let $G \subseteq A$ be a copy of $A(\tilde A_n)$ such that $A$ retracts onto $G$.
    Specifically, we have an inclusion map $i \colon G \hookrightarrow A$ and a projection $\pi \colon A \to G$ such that $\pi \circ i = \text{id}_{G}$.
    The induced maps $\pi_*$ and $i_*$ in homology satisfy $\pi_* \circ i_* = \text{id}_{H_*(G; \mathbb{Z})}$.
    In particular, the map $i_*$ is an injection of $H_*(G;\mathbb{Z})$ into $H_*(A; \mathbb{Z})$.

    \smallskip \noindent If $n = 1$, this is impossible because $H_1(G; \mathbb{Z}) \cong \mathbb{Z}^2$ and $H_1(A; \mathbb{Z}) \cong \mathbb{Z}$ (the first homology group can be easily computed from the defining presentation of an Artin group).
    If $n \geq 2$, this is impossible by Proposition~\ref{prop:affine-homology}, because $H_2(G;\mathbb{Z})$ contains a copy of $\mathbb{Z}$, whereas $H_2(A; \mathbb{Z})$ does not.
\end{proof}

\begin{proof}[Proof of Theorem~\ref{affArt_existeq}]
Let $A=A(\tilde A_n)$ with $n\geq4$, and suppose that $A$ and $B$ are existentially equivalent. If $B$ has type $\tilde D_m$ or $\tilde E_m$, Corollary~\ref{affArt_retr} makes $A$ a retract of $B$, contradicting Corollary~\ref{tildeADE}. If $B=A(\tilde A_m)$ with $m\geq4$, Corollary~\ref{tildeA_exequiv} gives $m=n$. The remaining simply laced cases are $m=2,3$. Corollary~\ref{affArt_retr} would give an embedding $A(\tilde A_n)\hookrightarrow A(\tilde A_m)$, contradicting monotonicity of integral cohomological dimension, since $\operatorname{cd}A(\tilde A_j)=j$. Conversely, isomorphic groups have the same existential theory.
\end{proof}

We think that, for $n\geq 5$ odd, the homology group $H_{n-1}(A(\tilde{A}_n); \mathbb{Z})$ has more torsion than $H_{n-1}(A(\tilde{B}_m); \mathbb{Z})$ and $H_{n-1}(A(\tilde{C}_m); \mathbb{Z})$. This would allow one to extend Corollary \ref{tildeADE} to the families $\tilde{B}_m$ and $\tilde{C}_m$ when $n$ is odd and $\geq 5$.
In the remaining cases, however, homology does not appear to provide an obstruction to $A(\tilde{A}_n)$ being a retract of $A(\tilde{B}_m)$ or $A(\tilde{C}_m)$, so new ideas would be required.

\end{document}